\documentclass{amsart}
\usepackage{amsmath, amsthm, amssymb}

\newtheorem{Theorem}{Theorem}[section]
\newtheorem{Definition}[Theorem]{Definition}
\newtheorem{Lemma}[Theorem]{Lemma}
\newtheorem{Proposition}[Theorem]{Proposition}

\newtheorem{Corollary}[Theorem]{Corollary}

\numberwithin{equation}{section}
\def\nt{\noindent}

\begin{document}

\title{A rescaled expansiveness for flows}


\author{Xiao Wen }
\address{School of Mathematics and System Science, Beihang University, Beijing 100191, China}
\curraddr{} \email{wenxiao@buaa.edu.cn}
\thanks{}

\author{Lan Wen}
\address{LMAM, School of Mathematical Sciences, Peking University, Beijing 100871, China}
\curraddr{} \email{lwen@math.pku.edu.cn}
\thanks{}



\date{}

\dedicatory{}

\begin{abstract}
We introduce a new
version of expansiveness for flows. Let $M$ be a compact Riemannian manifold without boundary and $X$ be
a $C^1$ vector field on $M$ that generates a flow $\varphi_t$ on $M$.  We call $X$ {\it rescaling expansive} on a compact invariant set $\Lambda$ of $X$ if for any $\epsilon>0$ there is $\delta>0$ such that, for any $x,y\in \Lambda$ and any time
reparametrization $\theta:\mathbb{R}\to \mathbb{R}$, if $d(\varphi_t(x), \varphi_{\theta(t)}(y)\le \delta\|X(\varphi_t(x))\|$ for all $t\in \mathbb R$, then $\varphi_{\theta(t)}(y)\in \varphi_{[-\epsilon, \epsilon]}(\varphi_t(x))$ for all $t\in \mathbb R$. We prove that every multisingular hyperbolic set (singular hyperbolic set in particular) is rescaling expansive and a converse holds generically.

\end{abstract}

\maketitle

\section{Introduction}
Expansiveness is a strong symbol of chaotic dynamics that has been studied extensively.
Recall for systems of discrete time, say for a homeomorphism $f$ of a compact metric space $M$, {\it expansiveness} states that there is $\delta>0$ such that, for any $x$ and $y$ in $M$,  $d(f^n(x), f^n(y))<\delta$ for all $n\in \mathbb Z$ implies $x=y$. In other words, expansiveness requests that any two different points $x$ and $y$ must get separated in a uniform distance $\delta$ at certain moment $n$. For systems of continuous time, the situation is quite different. As Bowen-Walters [BW] point out, for a flow $\varphi_t$ on $M$, any orbit itself is  a priori  ``non-expansive" because, for any $\delta> 0$, there is $\eta>0$ such that if $y=\varphi_s(x)$ with $s\in (-\eta, \eta)$ then $d(\varphi_t(x), \varphi_t(y))=d(\varphi_t(x), \varphi_s(\varphi_t(x)))<\delta$ for all $t\in \mathbb R$. Thus the best possible expansive property one could expect for a flow seems to be that for any $\epsilon>0$ there is $\delta>0$ such that if $y$ $\delta$-shadows $x$ then $y=\varphi_s(x)$ with $s\in (-\epsilon, \epsilon)$. To make the definition a conjugacy invariant and to rule out some pathological behavior one needs to allow time-reparametrizations. This leads to the following definition introduced by Komuro [Kom1]:

 A flow $\varphi_t$ is {\it expansive} on a compact invariant set $\Lambda$ of $\varphi_t$ if for any $\epsilon>0$ there is $\delta>0$ such that, for any $x$ and $y$ in $\Lambda$ and any surjective increasing continuous functions $\theta: \mathbb R\to \mathbb R$, if $d(\varphi_t(x), \varphi_{\theta(t)}(y)\le \delta$ for all $t\in \mathbb R$, then $\varphi_{\theta(t_0)}(y)\in \varphi_{[-\epsilon, \epsilon]}(\varphi_{t_0}(x))$ for some $t_0\in \mathbb R$. Komuro [Kom1]  proved that  the geometrical Lorenz attractor \cite{Lor}\cite{Gu} is expansive.  Araujo-Pacifico-Pujals-Viana [APPV] proved that every singular hyperbolic (see definition below) attractor in a $3$-dimensional manifold is expansive.

In this paper we introduce another version of expansiveness for flows, in which the shadowing condition $d(\varphi_t(x), \varphi_{\theta(t)}(y)\le\delta$ is rescaled by the flow speed:

\begin{Definition}
{\rm A flow $\varphi_t$ generated by a $C^1$ vector field $X$ is {\it rescaling expansive} on a compact invariant set $\Lambda$ if for any $\epsilon>0$ there is $\delta>0$ such that, for any $x,y\in \Lambda$ and any increasing continuous functions $\theta: \mathbb R\to \mathbb R$, if $d(\varphi_t(x), \varphi_{\theta(t)}(y)\le \delta\|X(\varphi_t(x))\|$ for all $t\in \mathbb R$, then $\varphi_{\theta(t)}(y)\in \varphi_{[-\epsilon, \epsilon]}(\varphi_t(x))$ for all $t\in \mathbb R$.}

\end{Definition}

Here $\theta$ is not assumed to be surjective. But we will see that, for small $\delta$, the shadowing condition $d(\varphi_t(x), \varphi_{\theta(t)}(y)\le \delta\|X(\varphi_t(x))\|$ forces $\theta$ to be surjective. Also note that in the definition if  $x$ is a singularity then $y=x$, and if $x$ is regular then  $y$ is regular if $\delta$ is small. Similarly, in the definition of expansiveness of Komuro, if one of the two points $x$ and $y$ is a hyperbolic singularity and if $\delta$ is small, then the other point must be the same singularity by the Hartman-Grobman theorem. Thus the expansive property for flows has  nontrivial behavior only  when both $x$ and $y$ are regular points.

The idea of rescaling the size of neighborhoods of a regular point by the flow speed comes from the classical work of Liao on standard systems of differential equations [L1, L2]. The recent paper of Gan-Yang \cite{GY}  extracts geometrically the ideas of Liao  to form an important tool in their work. See also \cite{HW}, \cite{SGW} and \cite{Y} for some relevant applications.
We remark that for nonsingular flows the two definitions, expansiveness and rescaled expansiveness,  are equivalent. (There is a discussion on the equivalence in the appendix at the end of this paper.) Nevertheless for flows with singularities we do not know if the two definitions imply one another.


In this paper we prove that every multisingular hyperbolic set,  singular hyperbolic set in particular, is rescaling expansive and a converse holds generically.
We first state the definition of singular hyperbolic set, which is introduced by Morales, Pacifico and Pujals \cite{MPP}, partly to characterize the celebrated geometrical Lorenz attractor \cite{Lor}\cite{Gu}.

Let $M$ be a $d$-dimensional compact  Riemannian manifold without  boundary and $X$ be a $C^1$ vector field on $M$. Denote $\varphi_t=\varphi^X_t$ the flow generated by $X$, and $\Phi_t=d\varphi_t:TM\to TM$
the  tangent flow of $X$. We call $x\in M$ a {\it singularity} of $X$ if $X(x)=0$. Denote ${\rm Sing}(X)$ the set of singularities of $X$. We call $x\in M$ a {\it regular point} if $x\in M\setminus {\rm Sing}(x)$.

Let $\Lambda$ be an invariant set of $X$. Let $C>1, ~ \lambda>0$ be given. We say a continuous $\Phi_t$-invariant splitting $T_\Lambda M=E\oplus F$ is a {\it $(C,\lambda)$-dominated splitting} with respect to $\Phi_t$ if
$$\|\Phi_t|_{E_x}\|\cdot\|\Phi_{-t}|_{F_{\varphi_t(x)}}\|<Ce^{-\lambda t}$$
for any $x\in \Lambda$ and $t>0$.

\begin{Definition}
{\rm Let $\Lambda$ be a compact invariant set of $X$. Let $C>1, ~ \lambda>0$ be given. We say that $\Lambda$ is {\it positively $(C,\lambda)$-singular hyperbolic} for
$X$ if there
is  a $(C,\lambda)$-dominated splitting
$T_\Lambda M=E\oplus F$ such that the following three conditions are
satisfied:

(1) the subbundle $E$ is $(C,\lambda)$-contracting with respect to $\Phi_t$,
that is, $$\|\Phi_t|_{E_x}\|<Ce^{-\lambda t},$$ for any $x\in\Lambda$
and $t>0$.

(2) the subbundle $F$ is $(C,\lambda)$-area-expanding with respect to $\Phi_t$, that is,
$$|\det(\Phi_{-t}|_{L})|<Ce^{-\lambda t},$$
for any $x\in \Lambda$ and any two dimensional subspace $L\subset F_x$
and any $t>0$,

(3) every singularity in $\Lambda$ is hyperbolic.

We say that $\Lambda$ is {\it singular hyperbolic} for $X$ if $\Lambda$ is
positively singular hyperbolic for $X$ or $-X$.}
\end{Definition}

We will prove that every singular hyperbolic set is rescaling expansive. This will be a corollary of Theorem A stating that every {\it multisingular hyperbolic set} is rescaling expansive. The notion of multisingular hyperbolic set, introduced recently by Bonatti-da Luz [BL], is  more general than the notion of singular hyperbolic set. (Proposition \ref{lemma32} explains that every singular hyperbolic set is multisingular hyperbolic.) First we recall the linear Poincar\'e flow defined on the normal bundle of $X$ over regular points of $X$. For  $x\in M\setminus {\rm Sing}(X)$, denote  the {\it normal space} of $X(x)$ to be  $$N_x=N_x(X)=\{v\in T_x M:v\perp X(x)\}.$$  Denote the {\it normal bundle} of $X$ to be
$$N=N(X)=\bigcup_{x\in M\setminus
{\rm Sing}(X)} N_x.$$ The {\it linear Poincar\'e flow } $\psi_t:N\to N$ of $X$  is then defined to be the orthogonal projection of $\Phi_t|_N$ to $N$, i.e.,
$$\psi_t(v)=\Phi_t(v)-\frac{\langle \Phi_t(v), X(\varphi_t(x))\rangle}{\|X(\varphi_t(x))\|^2}X(\varphi_t(x))$$
for any $v\in N_x$, where $\langle\cdot, \cdot\rangle$ denotes the Riemannian metric.

The notion of multisingular hyperbolicity is formulated using the {\it extended linear Poincar\'e flow} \cite{LGW}, a  ``compactification" of the usual linear Poincar\'e flow.
Denote $SM=\{e\in TM:\|e\|=1\}$ the unit sphere bundle of $M$ and $j:SM\to M$  the bundle projection defined by $j(e)=x$ if $e\in SM\cap T_xM$. Note that $SM$ is compact.
The tangent flow $\Phi_{t}$ induces a flow
$$\Phi_{t}^{\#}:SM \to SM$$
$$\Phi_{t}^{\#}(e)=\Phi_{t}(e)/\|\Phi_{t}(e)\|.$$
For any $e\in SM$, let
$$N_e=\{v\in T_{j(e)}M: v\perp e\}$$
be the normal space of $e$. Denote
$$N=N_{SM}=\bigcup_{e\in SM}N_e.$$
Then $N$ is a $d-1$ dimensional vector bundle over the base $SM$, irrelevant to vector fields. This bundle and the normal bundle $N=N(X)$ of a vector field $X$ over $M\setminus {\rm Sing}(X)$ are both abbreviated as $N$, which should not cause a confusion from the context. Define the {\it extended linear Poincar\'e flow } to be
$$\tilde\psi_t: N\to N$$
$$\tilde\psi_t(v)=\Phi_t(v)-\langle \Phi_t(v),\Phi_{t}^{\#}(e) \rangle\cdot \Phi_{t}^{\#}(e), \ \ \text{if} \ v\in N_e.$$
Thus $\tilde\psi_t$ covers the flow $\Phi_{t}^{\#}$ of $SM$, that is, $$\iota\circ \tilde\psi_t=\Phi_{t}^{\#}\circ \iota,$$
where $\iota: N_{SM}\to SM$ is the bundle projection.

If $e=X(x)/\|X(x)\|$ where $x\in M\setminus {\rm Sing}(X)$, then $N_e=N_x(X)$ and
 $$\Phi_{t}^{\#}(e)=X(\varphi_t(x))/\|X(\varphi_t(x))\|.$$ Hence
 $$\tilde\psi_t(v)=\Phi_t(v)-\langle \Phi_t(v), \frac{X(\varphi_t(x))}{\|X(\varphi_t(x))\|} \rangle\cdot \frac{X(\varphi_t(x))}{\|X(\varphi_t(x))\|}=\psi_t(v).$$  In other words, the extended linear Poincar\'e flow $\tilde\psi_t$ over the subset $\{X(x)/\|X(x)\|: x\in M\setminus {\rm Sing}(X)\}$ of $SM$ can be identified with the usual linear Poincar\'e flow $\psi_t$ over  $M\setminus {\rm Sing}(X)$.

 Let $\Lambda\subset M$ be a compact invariant set of $X$. Denote
 $$\tilde\Lambda=\overline{\{X(x)/\|X(x)\|: x\in \Lambda\setminus {\rm Sing}(X)\}}, $$ where the closure is taken in $SM$. The set $\tilde\Lambda$ is compact and $\Phi^{\#}_t$-invariant. Due to the parallel feature of vector fields near regular points, at every $x\in \Lambda\setminus {\rm Sing}(X)$, $\tilde\Lambda$ gives a single unit vector $X(x)/\|X(x)\|$. Thus in a sense $\tilde\Lambda$ is a ``compactification" of $\Lambda\setminus {\rm Sing}(X)$. At a singularity $x\in \Lambda\cap {\rm Sing}(X)$ however, $\tilde\Lambda$ usually gives a bunch of unit vectors.

Now we give the definition of multisingular hyperbolic set of Bonatti-da Luz \cite{BL1,BL2}. Let $\Lambda$ be a compact invariant set of $X$. A continuous function $h: \tilde\Lambda\times\mathbb{R}\to (0,+\infty)$ is called a  {\it cocycle} of $X$ on $\tilde\Lambda$ if  for any  $e\in \tilde\Lambda$ and any $s, t\in \mathbb R$,
$h(e, s+t)=h(e, s)\cdot h(\Phi_s^\#(e), t).$ We often write $h(e, t)$ as $h_t(e).$ Two most important examples of cocycles are $h_t(e)=\|\Phi_t(e)\|$ and $h_t(e)\equiv 1.$ A cocycle is called {\it pragmatical} with respect to a singularity $\sigma$ if there is an isolating neighborhood $U$ of $\sigma$ in $M$ such that if $e$ and $\Phi_t^\#(e)$ are both in $j^{-1}U$ then $h_t(e)=\|\Phi_t(e)\|$, and if $e$ and $\Phi_t^\#(e)$ are both outside $j^{-1}U$ then $h_t(e)=1$ (see \cite{BL1} for a good figure illustration). A {\it reparametrizing cocycle} is a (finite) product of pragmatical cocycles with disjoint isolating neighborhoods.

\begin{Definition}
{\rm Let $\Lambda$ be a compact invariant set of $X$. Let $C>1, ~ \lambda>0$ be given. We call $\Lambda$ a {\it $(C,\lambda)$-multisingular hyperbolic set} of $X$ if there is a $\tilde\psi_t$-invariant splitting $N_{\tilde\Lambda}=\Delta^s\oplus \Delta^u$ such that

(1) $\Delta^s\oplus \Delta^u$ is $(C,\lambda)$-dominated with respect to $\tilde\psi_t$, that is, $\|\tilde\psi_t|_{\Delta^s(e)}\|\cdot\|\tilde\psi_{-t}|_{\Delta^u(\Phi^\#_t(e))}\|<Ce^{-\lambda t}$
for any $e\in \tilde\Lambda$ and $t>0$.

(2) there is a reparametrizing  cocycle $h_t^s$ of $X$ on $\tilde\Lambda$ such that $\Delta^s$ is $(C,\lambda)$-contracting for $h_t^s\cdot \tilde\psi_t$, that is, $\|h_t^s(e)\cdot \tilde\psi_t(v)\|<Ce^{-\lambda t}\|v\|$ for any $e\in \tilde\Lambda$ and any $v\in \Delta^s(e)$ and $t>0$.

(3) there is a reparametrizing  cocycle $h_t^u$ of $X$ on $\tilde\Lambda$ such that $\Delta^u$ is $(C,\lambda)$-expanding for $h_t^u\cdot \tilde\psi_t$, that is, $\|h_{-t}^u(e)\cdot \tilde\psi_{-t}(v)\|<Ce^{-\lambda t}\|v\|$ for any $e\in \tilde\Lambda$ and any $v\in \Delta^u(e)$ and $t>0$.
}
\end{Definition}

The notion of multisingular hyperbolicity characterizes star flows. Recall a vector field $X$ is a {\it star system} if there is a neighborhood $\mathcal{U}$ of $X$ in $\mathcal{X}^1(M)$ such that, for every $Y\in\mathcal{U}$, every singularity and every periodic orbit of $Y$  is hyperbolic. Star flows play a fundamental role in the proof of the stability conjecture of Smale and Palis \cite{PS}, and the characterization of star flows has remained a remarkable problem for several decades. Bonatti-da Luz \cite{BL1, BL2} proved recently that a generic flow $X$ is a star flow if and only if every chain class of $X$ is multisingular hyperbolic. This solves generically the long standing problem of characterizing star flows.

 In this paper we investigate some aspects of multisingular hyperbolicity.
Here is a main result of this paper.

\vskip 0.2cm

\noindent {\bf Theorem A. } {\it
Let $\Lambda$ be a multisingular
hyperbolic set of a $C^1$ vector field $X$ on $M$. Then $\Lambda$
is rescaling expansive. In fact, there is $\epsilon_0>0$ such that for any $0<\epsilon\leq\epsilon_0$, any $x\in\Lambda$ and $y\in M$, and any increasing continuous functions $\theta: \mathbb R\to \mathbb R$, if
$d(\varphi_{\theta(t)}(y),\varphi_t(x))\le(\epsilon/3)\|X(\varphi_t(x))\|$ for all
$t\in\mathbb{R}$, then  $\varphi_{\theta(t)}(y)\in \varphi_{[-\epsilon, \epsilon]}(\varphi_t(x))$ for all $t\in \mathbb R$.
}

\vskip 0.2cm

Thus, for a multisingular hyperbolic set, the number $\delta$ in Definition 1.1 can be specified to be $\epsilon/3$.

On the other hand, a converse of Theorem A holds generically. That is, the rescaling expansiveness generically implies the multisingular hyperbolicity as stated in Theorem B below.

To state Theorem B we insert some definitions.  For $\delta>0$, we say that a sequence $\{(x_i,t_i):t_i\in M, t_i\geq 1\}_{a< i< b}$($-\infty\leq a<b\leq+\infty$) is a {\it $(\delta,1)$-pseudo orbit} or {\it $(\delta,1)$-chain} of $X$ if $d(\varphi_{t_i}(x_i),x_{i+1})<\delta$ for all $a<i<b-1$. Given $x,y\in M$, we say that $y$ is {\it chain attainable} from $x$ if for any $\delta>0$, there is a $(\delta,1)$-chain $\{(x_i,t_i)\}_{1\leq i\leq n}$, $n>1$,  such that $x_1=x, x_n=y$.  A compact invariant set $\Lambda$ is called {\it chain transitive} if every pair of points $x,y \in\Lambda$ are chain attainable from each other {\it through points of} $\Lambda$, that is, for any $\delta>0$, there are a $(\delta,1)$-chain $\{(x_i,t_i)\}_{1\leq i\leq n}$  with $x_1=x, x_n=y$ and a  $(\delta,1)$-chain $\{(y_i,t_i)\}_{1\leq i\leq n}$ with $y_1=y, y_n=x
$, where all $x_i,y_i$ are in $\Lambda$.

A compact invariant set $\Lambda$ of $X$ is called {\it isolated} if there is a neighborhood $U\subset M$ of $\Lambda$ such that
$$\Lambda=\bigcap_{t\in \mathbb{R}}\varphi_t(U).$$ An isolated invariant set $\Lambda$ is called {\it locally star} for $X$ if there are a neighborhood $\mathcal{U}$ of $X$ in $\mathcal{X}^1(M)$ and a neighborhood $U$ of $\Lambda$ in $M$ such that, for every $Y\in\mathcal{U}$, every singularity and every periodic orbit of $Y$ that is contained (entirely) in $U$ is hyperbolic.

Let $\mathcal{X}^1(M)$ be the space of $C^1$ vector fields endowed with the $C^1$ topology. A subset $\mathcal{R}\subset \mathcal{X}^1(M)$ is called  {\it residual} if it is an intersection of countably open and dense subset of $\mathcal{X}^1(M)$.
\vskip 0.2cm

\noindent {\bf Theorem B. } {\it
There is a residual set $\mathcal{R}\subset \mathcal{X}^1(M)$ such that for any $X\in\mathcal{R}$ and any isolated chain transitive set $\Lambda$, the following three conditions are equivalent:

\begin{enumerate}
\item $\Lambda$ is rescaling expansive for $X$.
\item $\Lambda$ is locally star for $X$.
\item $\Lambda$ is multisingular hyperbolic for $X$.

\end{enumerate}}

\vskip 0.2cm

 \section{Time-reparametrizations}

 A basic tool to what follows will be  a ``uniform relative" version of the classical flowbox theorem. We first work with a Euclidean space.

 Let $\bar{X}$ be a $C^1$ vector field on $\mathbb{R}^n$ with $\|D\bar X(x)||\leq L$ for all $x\in\mathbb{R}^n$, where $L>0$ is a constant, and $\bar{\varphi}_t$ be the flow generated by $\bar{X}$. Here we use the notations of $\bar X$ and $\bar \varphi_t$ with a bar just to distinguish from the notations of the vector filed $X$ and the flow $\varphi_t$ on the manifold $M$ below.

 For every regular point $x\in\mathbb{R}^n\setminus {\rm Sing}(\bar{X})$ and every $r>0$, denote by $$\bar{U}_x(r\|\bar X(x)\|)=\{v+t\bar{X}(x): v\in N_x, \|v\|\leq r\|\bar X(x)\|, |t|\leq r\}$$
the {\it tangent box of relative size} $r$ at $x$, where $N_x$ denotes the normal space to the span of $\bar X(x)$.

Note that the size of the tangent box $\bar{U}_x(r\|\bar X(x)\|)$  is $r\|\bar X(x)\|$ but not $r$, and this is why we have called $r$ the {\it relative size} of the box, that is, the size relative to the flow speed $\|\bar X(x)\|$.

Define the {\it flowbox map} $F_x$ of $\bar X$ at $x$ to be
 $$F_x:\bar{U}_x(r\|\bar X(x)\|)\to \mathbb{R}^n$$
$$F_x(v+t\bar X(x))=\bar{\varphi}_t(x+v).$$
Thus, for every $v\in N_x$ with $\|v\|\leq r\|\bar X(x)\|$, $F_x$ maps the line interval $v+[-r\|\bar X(x)\|, ~ r\|\bar X(x)\|]$ onto the orbital arc $\{\bar{\varphi}_t(x+v):~ |t|\le r\}$. Note that
$F_x(0)=x.$

  Recall $m(A)$ denotes the {\it mininorm} of a linear operator $A$, i.e., $$m(A)=\inf \{\|A(v)\|: v\in \mathbb R^n, ~ \|v\|=1\}.$$

\begin{Proposition}\label{flowbox1}
Let $\bar{X}$ be a $C^1$ vector filed on $\mathbb{R}^n$ such that $\|D\bar X(x)||\leq L$ for all $x\in\mathbb{R}^n$. There is $r_0>0$ such that, for any regular point $x$ of $\bar{X}$, $F_x:\bar{U}_x(r_0\|\bar X(x)\|)\to \mathbb{R}^n$ is an embedding whose image contains no singularities of $\bar X$, and $m(D_pF_x)\geq 1/2$ and
$\|D_pF_x\|\leq2$ for every $p\in\bar{U}_x(r_0\|\bar X(x)\|)$.

\end{Proposition}

We call  the image $F_x(\bar{U}_x(r_0\|\bar X(x)\|))$ a {\it flowbox} of $\bar X$ of {\it relative size} $r_0$ at $x$.  Proposition \ref{flowbox1} says that, although the set of regular points  of $\bar X$ is non-compact,  the  relative size $r_0$ of flowboxes  for all regular points  can be chosen uniform.

The  idea and the term of ``flowbox" are classical. See for instance Pugh-Robinson [PR] for the definition of flowbox. The ``uniform relative" version like Proposition \ref{flowbox1} is probably new.

\begin{proof}
Since $\sup\{\|D\bar{X}\|\}\leq L$,  the vector field $\bar{X}$ on $\mathbb R^n$ is Lipschitz with a Lipschitz constant $L$. Hence if $$\|y-x\|\leq\frac{1}{4L}\|\bar{X}(x)\|$$ then
$$\|\bar{X}(y)-\bar{X}(x)\|\leq L\|x-y\|\leq\frac{1}{4}\|\bar{X}(x)\|.\ \ \ \ \ \ \ \ \ \ \ \ \ \ \ \ \ \   (*)$$
In particular,  if $$y-x\in \bar{U}_x(\frac 1{4L}\|\bar X(x)\|)\cap N_x$$ then  $\bar{X}(y)\not=0$. This means that, for every regular point $x\in \mathbb R^n$, the flowbox
$F_x(\bar{U}_x(r_0\|\bar X(x)\|))$  contains no singularities if $r_0\le \frac 1 {4L}$.

\bigskip
{\noindent\bf Claim 1. } If $\|y-x\|\leq \frac{1}{8L}\|\bar{X}(x)\|$ and $|t|\leq\frac{1}{10L}$, then $\|\bar{\varphi}_t(y)-x\|\leq \frac{1}{4L}\|\bar{X}(x)\|$.

\bigskip
{\noindent\it Proof. }
Suppose for the contrary there are $y_0$ and $t_0$ with
$$\|y_0-x\|\leq\frac{1}{8L}\|\bar{X}(x)\|, ~~~ |t_0|\leq
\frac{1}{10L}$$ but
$$\|\bar{\varphi}_{t_0}(y_0)-x\|>\frac{1}{4L}\|\bar{X}(x)\|.$$ Without loss of
generality we assume $t_0>0$. Then there is $t_1$ with $0<t_1<t_0$ such
that $$\|\bar{\varphi}_{t_1}(y_0)-x\|=\frac{1}{4L}\|\bar{X}(x)\|$$
but
$$\|\bar{\varphi}_t(y_0)-x\|<\frac{1}{4L}\|\bar{X}(x)\|$$ for every $0<t<t_1$. By $(*)$,
 $$\|\bar{X}(\bar{\varphi}_t(y_0))\|\leq
\frac{5}{4}\|\bar{X}(x)\|$$ for all $0\leq t\leq t_1$. Then
$$\|\bar{\varphi}_{t_1}(y_0)-y_0\|=\|\int_{0}^{t_1}\bar{X}(\bar{\varphi}_t(y_0))dt\|\leq\int_{0}^{t_1}\|\bar{X}(\bar{\varphi}_t(y_0))\|dt$$
$$\leq\int_{0}^{t_1}\frac{5}{4}\|\bar{X}(x)\|dt=\frac{5}{4}\|\bar{X}(x)\|t_1<\frac{5}{4}\|\bar{X}(x)\|\frac{1}{10L}=\frac{1}{8L}\|\bar{X}(x)\|.$$
Here the strict inequality is guaranteed by
  $$t_1<t_0, ~ ~ t_0\le \frac 1{10L}.$$ Hence
$$\|\bar{\varphi}_{t_1}(y_0)-x\|\leq\|\bar{\varphi}_{t_1}(y_0)-y_0\|+\|y_0-x\|<\frac{1}{4L}\|\bar{X}(x)\|,$$
contradicting
$$\|\bar{\varphi}_{t_1}(y_0)-x\|=\frac{1}{4L}\|\bar{X}(x)\|.$$ This proves Claim 1.

\bigskip
Now let  $$r_0=\frac{1}{10L}.$$ We verify that $r_0$ satisfies the requirement of the proposition.

\bigskip
\nt {\bf Claim 2.} For any $p=v+t\bar{X}(x)\in \bar{U}_x(r_0\|\bar X(x)\|)$, $\|D_pF_x-id\|\le 1/2.$

\bigskip
{\noindent\it Proof.}  A straightforward computation of the directional derivative of $F_x$ at $p$ along the direction $\bar X(x)$ gives
$${D_pF_x} \cdot {\bar{X}(x)}=\bar{X}(F_x(p)).$$
Now
$|t|\leq\frac{1}{10L}$ and $\|v\|<\frac{1}{8L}\|\bar{X}(x)\|$ hence, by Claim 1,  $$\|F_x(p)-x\|=\|\bar{\varphi}_t(x+v)-x\|\leq\frac{1}{4L}\|\bar{X}(x)\|.$$
Since $\bar X$ has Lipschitz constant $L$, we have $$\|\bar{X}(F_x(p))-\bar{X}(x)\|\leq \frac{1}{4}\|\bar{X}(x)\|.$$ That is,
$$\|{D_pF_x} \cdot {\bar{X}(x)}-\bar{X}(x)\|\le \frac{1}{4}\|\bar{X}(x)\|.$$
Or
$$\|D_pF_x|_{<\bar{X}(x)>}-id\|\leq 1/4.$$

Likewise, for any $u\in N_x$, a straightforward computation of the directional derivative of $F_x$ at $p$ along the direction $u$ gives
$$D_pF_x\cdot u=D_{x+v}{\varphi}_t\cdot u.$$ Then
$$\|D_pF_x|_{N_x}-id\|\leq |e^{Lt}-1|\leq e^{1/10}-1<1/4.$$ Thus, for every $p\in \bar{U}_x(r_0\|\bar X(x)\|)$, $$\|D_pF_x-id\|\leq\|D_pF_x|_{<\bar{X}(x)>}-id\|+\|D_pF_x|_{N_x}-id\|<1/2.$$
This proves Claim 2.
\bigskip

In particular,  for every $p\in \bar{U}_x(r_0\|\bar X(x)\|)$, $D_pF_x$ is a linear isomorphism. By the inverse function theorem, $F_x$ is a local diffeomorphism. To prove that $F_x$ is an embedding it suffices to prove $F_x$ is injective, i.e., to prove that for any $z\in \mathbb R^n$ there is at most one $y\in \bar{U}_x(r_0\|\bar X(x)\|)$ such that $F_x(y)=z.$ By the generalized mean value theorem,
Claim 2 gives
 $${\rm Lip}(F_x-id)\le 1/2.$$  Then we can write $$F_x=id+\phi,$$ where $$\phi:\bar{U}_x(r_0\|\bar X(x)\|)\to \mathbb R^n$$ is Lipschitz with ${\rm Lip}(\phi)\le 1/2.$ We need to prove that, for any $z\in \mathbb R^n$, $F_x(y)=y+\phi(y)=z$ has at most one solution for $y$ or, equivalently,  $y=z-\phi(y)$ has at most one solution for $y$.
 Define $$T=T_z: \bar{U}_x(r_0\|\bar X(x)\|)\to \mathbb R^n$$ to be $$T(y)=z-\phi(y).$$ It suffices to prove that $T$ has at most one fixed point. It is sufficient to verify that $T$ is a contraction mapping. This is straightforward because
 $$\|T(y)-T(y')\|\le {\rm Lip}(\phi)\|y-y'\|\le \frac 12\|y-y'\|.$$ This proves that $F_x$ is an embedding. Clearly,
 $$m(D_pF_x)\geq 1/2, ~~ \|D_pF_x\|\leq2$$
 for every $p\in \bar{U}_x(r_0\|\bar X(x)\|)$. This ends the proof of Proposition \ref{flowbox1}.
\end{proof}

Now we come back to our manifold. As usual, denote
$$T_xM(r)=\{v\in T_xM: \|v\|\le r\},$$ $$B_r(x)=\exp_x(T_xM(r)),$$ $$N_x(r)=\{v\in N_x:\|v\|\le r\}.$$ By the compactness of $M$ and the $C^1$ smoothness of $X$,
there are constants $L>0$ and $a>0$ such that for any $x\in M$ the
vector fields
$$\bar{X}=(\exp_x^{-1})_*(X|_{B_{a}(x)})$$ in
$T_xM(a)$ are locally Lipschitz vector fields with Lipschitz constant $L$. We call $L$ a {\it local Lipschitz constant} of $X$. We may assume  $$m(D_p\exp_x)>2/3, ~ \|D_p\exp_x\|<3/2$$ for any $p\in T_xM(a)$. For every $x\in M\setminus {\rm Sing}(x)$, denote  $$U_x(r\|X(x)\|)=\{v+tX(x)\in T_xM: v\in N_x, \|v\|\leq r\|X(x)\|, |t|\leq r\}$$ the tangent box of relative size $r$ at $x$. Define a $C^1$ map  $$F_x:U_x(r\|X(x)\|)\to M$$ to be $$F_x(v+tX(x))=\varphi_t(\exp_x(v)).$$ Then Proposition \ref{flowbox} gives directly the following proposition whose proof is omitted.

\begin{Proposition}\label{flowbox}
For any $C^1$ vector field $X$ on $M$, there is $r_0>0$ such that for any regular point $x$ of ${X}$, $F_x:{U}_x(r_0\|X(x)\|)\to M$ is an embedding whose image contains no singularities of $X$, and $m(D_pF_x)\geq 1/3$ and
$\|D_pF_x\|\leq3$ for every $p\in {U}_x(r_0\|X(x)\|)$.

\end{Proposition}

Here in the statement the constant 2 is  changed to 3 because of the involvement of the exponential maps ${\rm exp}_x$  in the proof.

Now we analyze the time-reparametrizations $\theta$ in the definition of the rescaled expansiveness. We will see that if $\delta$ is sufficiently small, the rescaled shadowing condition $$d(\varphi_{\theta(t)}(y),\varphi_t(x))\le\delta\|X(\varphi_t(x))\| ~~ {\rm for ~all} ~ t\in \mathbb R$$ of Definition 1.1 will force  $\theta$ to be nearly a translation.
The key to the proof is to control the time-difference $|t|$ by the distance $d(x,\varphi_t(x))$.
We know by continuity that if $|t|$ is small then  $d(x,\varphi_t(x))$ is small. The converse is not true if, for instance, $x$ is a singularity or $x$ is periodic and $t$ is  the period. Nevertheless in some situations a converse could be true. The next lemma states such a converse: in some situations $d(x,\varphi_t(x))\le \delta\|X(x)\|$ implies $|t|\le3\delta$. This will play a crucial role in the proof of Lemma \ref{proposition23}.

In what follows $r_0$ always denotes the constant given in Proposition \ref{flowbox}. Ideas  of Komuro [Kom2] are helpful to the rest part of this section.

\begin{Lemma} \label{lemma21}
Let  $x\in M\setminus {\rm Sing}(X)$ be given.

$(1)$  For any $0<\delta\le r_0/3$ and $t\in [-r_0, r_0]$,  $d(x,\varphi_t(x))\le \delta \|X(x)\|$ implies $|t|\le3\delta$.

$(2)$  For any $0<\delta\le r_0/3$, $\varphi_{[0,t]}(x)\subset B(x, \delta\|X(x)\|)$ implies $|t|\le 3\delta$.

\end{Lemma}

\begin{proof}
 (1) Since $F_x(0_x)=x$ and $m(D_pF_x)\geq 1/3$ and $\|D_pF_x\|\leq 3$ for every $p\in {U}_x(r_0\|X(x)\|)$,  we have $$F_x(U_x(r_0\|X(x)\|))\supset B(x, (r_0/3)\|X(x)\|).$$  Assume  $0<\delta\le r_0/3$, $t\in [-r_0, r_0]$, and $d(x,\varphi_t(x))\le \delta \|X(x)\|$. Take a geodesic $\gamma$ connecting $x$ and $\varphi_t(x)$. Then $\gamma\subset B(x, (r_0/3)\|X(x)\|)$. Since $t\in [-r_0, r_0]$, $F_x^{-1}(\varphi_t(x))=tX(x)$. Then  $F_x^{-1}(\gamma)$ is a curve in $U_x(r_0\|X(x)\|)$ connecting $0$ and $tX(x)$. Hence $$\|tX(x)\|\leq l(F_x^{-1}(\gamma))\leq 3l(\gamma)= 3d(x, \varphi_t(x))\le 3\delta\|X(x)\|.$$ Thus $|t|\le 3\delta$.

 (2) Assume $0<\delta\le r_0/3$ and $\varphi_{[0,t]}(x)\subset B(x, \delta\|X(x)\|)$. To prove $|t|\le 3\delta$, by (1), it suffices to verify $t\in [-r_0, r_0]$. Suppose $t\notin [-r_0, r_0]$. Without loss of generality suppose $t>r_0$. Take $s\in (r_0, t)$ slightly larger than $r_0$. Then $\varphi_{s}(x)\notin B(x, (r_0/3)\|X(x)\|)$, contradicting $\varphi_{[0,t]}(x)\subset B(x, \delta\|X(x)\|)$.  This  proves Lemma \ref{lemma21}.
\end{proof}
\vskip 0.2cm
\nt {\bf Remark}. In the proof of Lemma \ref{lemma21}, without the condition $t\in [-r_0, r_0]$, $F_x^{-1}(\varphi_t(x))$ may not be equal to $tX(x)$. For instance this is the case when $x$ is periodic and $t$ is the period of $x$.
\vskip 0.2cm
\begin{Lemma}\label{proposition23}
For any $\epsilon>0$ there is $\delta>0$ such that, for any $x\in M\setminus {\rm Sing}(X)$, any $y\in M$ and any $T\in[r_0/2,r_0]$, if there is an increasing continuous function $\theta:[0,T]\to\mathbb{R}$ such that $d(\varphi_t(x),\varphi_{\theta(t)}(y))\le\delta\|X(\varphi_t(x))\|$ for all $t\in[0,T]$, then $|\theta(T)-\theta(0)-T|\le\epsilon T$.
\end{Lemma}

\begin{proof}

Let $L$ be a local Lipschitz constant given in the paragraph right before Proposition \ref{flowbox}. First we recall two formulas from ODE about the continuous dependence of solutions with respect to initial conditions:

(1) For any $x\in M\setminus {\rm Sing}(X)$ and $t\in \mathbb{R}$, $$\frac{\|X(\varphi_t(x))\|}{\|X(x)\|}\in[e^{-L|t|}, e^{L|t|}].$$

(2) $d(\varphi_t(x),\varphi_t(y))\le e^{L|t|}d(x, y).$

We also fix a fact that can be proved like the inequality $(*)$ in the proof of Proposition 2.1:

\bigskip
\nt {\bf Fact.} {\it There is $c>0$ such that for any $z,z'\in M\setminus {\rm Sing}(X)$, if $d(z,z')<c\|X(z)\|$ then
$(1/2)\|X(z)\|<\|X(z')\|<2\|X(z)\|.$}

\bigskip
Now let $\epsilon>0$ be given. Let
$$\delta=\min\{\frac{r_0}{6e^{2Lr_0}}, ~ \frac{c}{18e^{2Lr_0}}, ~ \frac{\epsilon r_0}{12(3+18e^{2Lr_0})}\}.$$ Here the three expressions are just some rough estimates that will work.

Assume we are given $x\in M\setminus {\rm Sing}(X)$, $y\in M$, $T\in[r_0/2,r_0]$ and an increasing continuous function $\theta:[0,T]\to \mathbb{R}$ such that $$d(\varphi_t(x),\varphi_{\theta(t)}(y))\le\delta\|X(\varphi_t(x))\|$$ for all $t\in[0,T]$. We prove $|\theta(T)-\theta(0)-T|\le\epsilon T.$ Replacing $\theta$ by $\eta$ with $\eta(t)=\theta(t)-\theta(0)$ if necessary, we  assume $\theta(0)=0$.
Thus we prove $$|\theta(T)-T|\le\epsilon T.$$
Note that $\theta(0)=0$ implies $$d(x, y)\le \delta\|X(x)\|.$$
Most of the proofs will be to estimate the distance $d(\varphi_{\theta(T)}(y),\varphi_T(y))$. At last we will convert it to the time-difference $|\theta(T)-T|$, using Lemma  \ref{lemma21}.

First assume $\theta(T)\leq T$. Then
$$d(\varphi_{\theta(T)}(y),\varphi_T(y))\leq d(\varphi_{\theta(T)}(y),\varphi_T(x))+d(\varphi_T(x),\varphi_T(y))$$
$$\leq \delta\|X(\varphi_T(x))\|+e^{LT}\delta\|X(x)\|$$
$$\leq \delta\|X(\varphi_T(x))\|+e^{2LT}\delta\|X(\varphi_T(x))\|$$
$$=(1+e^{2LT})\delta\|X(\varphi_T(x))\|.$$
Since $$d(\varphi_T(x),\varphi_T(y))\le e^{2LT}\delta\|X(\varphi_T(x))\|,$$ and since  $$e^{2LT}\delta\leq c$$ (by the choice of $\delta$), by the above Fact,  $$\|X(\varphi_T(y))\|>1/2\|X(\varphi_T(x))\|.$$ Then
$$d(\varphi_{\theta(T)}(y),\varphi_T(y))\le 2(1+e^{2LT})\delta\|X(\varphi_T(y))\|.$$
Since $T\in[r_0/2,r_0]$, and since $\theta$ is increasing, $\theta(0)=0$, and $\theta(T)\leq T$, we have $|\theta(T)-T|\in [-r_0, r_0]$. By the choice of $\delta$, $$2(1+e^{2LT})\delta\le\frac{\epsilon r_0}{6}.$$ Then by the first part of Lemma \ref{lemma21}, $$|\theta(T)-T|\le\frac{\epsilon r_0}{2}\leq \epsilon T.$$

Now assume $\theta(T)>T$. There is $0\leq S\leq T$ such that $\theta(S)=T$. Then
$$d(\varphi_T(x),\varphi_{S}(x))\leq d(\varphi_T(x),\varphi_T(y))+d(\varphi_T(y),\varphi_{S}(x))$$
$$\leq e^{2LT}\delta\|X(\varphi_T(x))\|+\delta\|X(\varphi_{S}(x))\|$$
$$\le e^{2LT}\delta\|X(\varphi_T(x))\|+e^{LT}\delta\|X(\varphi_T(x))\|$$
$$=(e^{2LT}+e^{LT})\delta\|X(\varphi_T(x))\|.$$
Here we have used the fact $$\|X(\varphi_S(x))\|\le e^{L(T-S)}\|X(\varphi_{T}(x))\|\leq e^{LT}\|X(\varphi_T(x))\|.$$
By the choice of $\delta$, $$(e^{2LT}+e^{LT})\delta\le r_0/3.$$ Then by the first part of Lemma \ref{lemma21}, $$0\leq T-S\le 3(e^{2LT}+e^{LT})\delta\le 6e^{2LT}\delta.$$
Here we replace $3(e^{2LT}+e^{LT})$ by $6e^{2LT}$ just to shorten the expression. Note that by the choice of $\delta$, $$6e^{2LT}\delta\le r_0.$$    Consider the flowbox $U_{\varphi_T(x)}(r_0\|X(\varphi_T(x))\|)$ around $\varphi_T(x)$. By the definition of the flowbox map and Proposition \ref{flowbox}, for every $s\in [S,T]$, $$d(\varphi_s(x),\varphi_T(x))\le 3(T-s)\|X(\varphi_T(x))\|$$ $$\leq 3(T-S)\|X(\varphi_T(x))\|\le 18e^{2LT}\delta\|X(\varphi_T(x))\|.$$
Now for every $t\in[T,\theta(T)]$, take $s\in[S,T]$ such that $\theta(s)=t$. Then
$$d(\varphi_{\theta(T)}(y),\varphi_t(y))\leq d(\varphi_{\theta(T)}(y),\varphi_T(x))+d(\varphi_T(x),\varphi_{s}(x))+d(\varphi_{s}(x),\varphi_t(y))$$
$$\le\delta\|X(\varphi_T(x))\|+18e^{2LT}\delta\|X(\varphi_T(x))\|+\delta\|X(\varphi_{s}(x))\|.$$
By the choice of $\delta$, $$18e^{2LT}\delta\le c.$$ Then $\|X(\varphi_{s}(x))\|\le 2\|X(\varphi_T(x))\|.$ Therefore, for every $t\in[T,\theta(T)]$,
$$d(\varphi_{\theta(T)}(y),\varphi_t(y))\le(1+18e^{2LT})\delta\|X(\varphi_T(x))\|+2\delta\|X(\varphi_T(x))\|$$
$$=(3+18e^{2LT})\delta\|X(\varphi_T(x))\|$$
$$\le2(3+18e^{2LT})\delta\|X(\varphi_{\theta(T)}(y))\|.$$
Here we have used the fact  $d(\varphi_{\theta(T)}(y),\varphi_T(x))\le \delta\|X(\varphi_T(x))\|$ and hence the fact $\|X(\varphi_{\theta(T)}(y))\|\ge(1/2)\|X(\varphi_T(x))\|.$
Now by the choice of $\delta$, $$2(3+18e^{2LT})\delta\le \frac{\epsilon r_0}{6}.$$ Then by the second part of Lemma \ref{lemma21}, $$|\theta(T)-T|\le\frac{\epsilon r_0}{2}\leq \epsilon T.$$ This ends the proof of Lemma \ref{proposition23}.
\end{proof}

\begin{Lemma}\label{lemma23}
For any $\epsilon>0$ there is $\delta>0$  such that,
for any $x\in M\setminus {\rm Sing}(X)$, any $y\in M$, and any $T\geq r_0$, if there
is an increasing continuous function $\theta:[0,T]\to\mathbb{R}$  such that
$d(\varphi_t(x),\varphi_{\theta(t)}(y))\le\delta\|X(\varphi_t(x))\|$ for
all $t\in[0,T]$, then $|\theta(T)-\theta(0)-T|\le\epsilon T$.
\end{Lemma}

\begin{proof}
This is a corollary of Lemma \ref{proposition23}. We prove the case $T\ge r_0$.
Divide $[0,T]$ into several intervals as
$$0=T_0<T_1<\cdots<T_{n-1}<T_n=T$$
with $r_0/2\leq T_i-T_{i-1}<r_0$. Then we choose $\delta>0$ by
Lemma \ref{proposition23}. Let $\theta:[0,T]\to\mathbb{R}$ be an increasing continuous function such that
$$d(\varphi_t(x),\varphi_{\theta(t)}(y))\le\delta\|X(\varphi_t(x))\|$$ for
all $t\in[0,T]$. Without loss of generality we assume
$\theta(0)=0$. Then for any $1\leq i\leq n$,
$$|\theta(T_i)-\theta(T_{i-1})-(T_i-T_{i-1})|\le\epsilon(T_i-T_{i-1}).$$ Hence
we have  $|\theta(T)-T|\le\epsilon T$. This ends the proof of Lemma \ref{lemma23}.
\end{proof}

\begin{Corollary}\label{cor} There is $\delta>0$  such that,
for any $x\in M\setminus {\rm Sing}(X)$ and any $y\in M$, if there
is an increasing continuous function $\theta:\mathbb R\to\mathbb{R}$  such that
$d(\varphi_t(x),\varphi_{\theta(t)}(y))\le\delta\|X(\varphi_t(x))\|$ for
all $t\in \mathbb R$, then $\theta$ is surjective.
\end{Corollary}

\begin{proof} Fix $0<\epsilon<1$. Let $\delta>0$ be the corresponding number guaranteed by Lemma \ref{lemma23}. Assume we are given $x\in M\setminus {\rm Sing}(X)$, $y\in M$, and an increasing continuous function $\theta:\mathbb R\to\mathbb{R}$  such that
$d(\varphi_t(x),\varphi_{\theta(t)}(y))\le\delta\|X(\varphi_t(x))\|$ for
all $t\in \mathbb R$. By Lemma \ref{lemma23},  for any $T\ge r_0$, $|\theta(T)-\theta(0)-T|\le\epsilon T$. Then $\theta(T)\to \infty$ as $T\to\infty$. Symmetrically, the same argument proves $\theta(T)\to -\infty$ as $T\to-\infty$. Thus
$\theta:\mathbb R\to \mathbb R$ is surjective.
\end{proof}

\section{Sectional Poincar\'e maps}

 In this section we discuss the sectional Poincar\'e maps. Let $X$ be a $C^1$ vector field on $M$ as before. By Proposition \ref{flowbox}, the flowbox $F_x(U_x(r_0\|X(x)\|))$ contains a ball of radius $\frac{r_0}{3}\|X(x)\|$ centered at $x$. For any $y\in B_{\frac{r_0}{3}\|X(x)\|}(x)$, if $y=F_x(v,t)$, then define $P_x(y)=v$. In other words, define a map
 $$P_x: B_{\frac{r_0}{3}\|X(x)\|}(x)\to N_x$$ to be
$$P_x=\pi_x\circ F_x^{-1},$$ where $\pi_x$ denotes the orthogonal projection of $T_x M$ to $N_x$.  Since $\pi_x$ has norm $\le 1$, we have  $$\|DP_x|_{B_{\frac{r_0}{3}\|X(x)\|}(x)}\|\leq 3.$$

For any $t\in\mathbb{R}$, let $$r_1=r_1(t)=e^{-2L|t|}\frac{r_0}{3},$$ where $L$ is chosen as in the previous paragraphs with the property $\sup\{\|DX\|\}<L$. Since $$\|X(\varphi_t(x))\|\le e^{Lt}\|X(x)\|,$$ we have $$\varphi_t(B_{r_1\|X(x)\|}(x))\subset B_{e^{L|t|}r_1\|X(x)\|}(\varphi_t(x))$$$$\subset B_{e^{2L|t|}r_1\|X(\varphi_t(x))\|}(\varphi_t(x))=B_{\frac{r_0}{3}\|X(\varphi_t(x))\|}(\varphi_t(x)).$$
Hence for any $t\in \mathbb{R}$, we can define a map
$$P_{x,t}:N_x(r_1\|X(x)\|)\to N_{\varphi_t(x)}$$ to be $$P_{x,t}=P_{\varphi_t(x)}\circ\varphi_t\circ\exp_x,$$  called the {\it sectional Poincar\'e map }at $x$ of time $t$ (\cite{GY}). Note that $P_{x,t}$ and  $P_x$ are  different maps.

The sectional Poincar\'e map $P_{x,t}$ is defined in Gan-Yang \cite{GY} using holonomy maps generated by orbit arcs. The definition here using flowboxes is equivalent but formally slightly different.

The following proposition presents some uniform ( in a relative sense) property about the family of the derivatives  $D_vP_{x,T}$ of the sectional Poincar\'e maps $P_{x,T}$ at $v\in N_x(r_1\|X(x)\|)$. Note that, at the origin $0_x$ of $N_x$,
$$D_{0_x}P_{x,T}=\psi_T|_{N_x}.$$

\begin{Proposition}\label{prop3.1}\cite{GY}
The family of sectional Poincar\'e maps $\{P_{x,T}\}$ has the following
properties:
\begin{enumerate}
\item $\|D_vP_{x,T}\|$ is uniformly bounded in the following sense: for any $T\in\mathbb{R}$, there is $K>0$ such that $\|D_vP_{x,T}\|\leq K$ for any $x\in M\setminus {\rm Sing}(X)$ and any $v\in N_x(r_1\|X(x)\|)$.
\item $D_vP_{x,T}$ is uniformly continuous in the following sense: Given $T\in\mathbb{R}$, for any $\epsilon>0$ there is $\delta>0$ such that for any $x\in M\setminus {\rm Sing}(X)$ and any $v,v'\in N_x(r_1\|X(x)\|)$ with $\|v-v'\|<\delta\|X(x)\|$, $\|D_vP_{x,T}-D_{v'}P_{x,T}\|<\epsilon$.
\end{enumerate}
\end{Proposition}

Here $x$ ranges over the non-compact set $M\setminus {\rm Sing}(X)$ and $v$ ranges over a neighborhood of $0_x$ in $N_x$ of uniform relative size $r_1=r_1(T)$. This proposition extracts some old ideas from the work of Liao [L1, L2] and plays an important tool in the recent work of Gan-Yang \cite{GY}. See also \cite{HW}, \cite{SGW} and \cite{Y} for some relevant applications. Since our definition of sectional Poincar\'e maps $\{P_{x,T}\}$ is formally slightly different from the one given originally in \cite{GY}, we give a sketch of the proof for Proposition \ref{prop3.1}.

\begin{proof}
The proof for (1) is immediate because
$$\|D_vP_{x,T}\|=\|D_v(P_{\varphi_T(x)}\circ\varphi_T\circ\exp_x)\|$$
$$\leq\|D_{\varphi_T(\exp_x(v))}P_{\varphi_T(x)}\|\cdot\|D_{\exp_x(v)}\varphi_T\|\cdot\|D_v\exp_x\|$$
$$\leq 3\cdot e^{L|T|}\cdot\frac{3}{2}=\frac{9}{2}e^{L|T|}$$
for any $v\in N_x(r_1\|X(x)\|)$.

The key to the proof of (2) is that the derivatives of the flowbox map $F_x$ are uniformly continuous in the following relative sense. Modulo exponential maps ${\rm exp}_x$, assume we work in a Euclidean space.

\bigskip
{\noindent\bf Claim.} {\it For any $\epsilon>0$, there is $\delta>0$ such that for any $x\in M\setminus Sing(X)$ and any $p,q\in U_x(r_0\|X(x)\|)$, if $d(p,q)<\delta\|X(x)\|$, then $\|D_pF_x-D_qF_x\|<\epsilon$.}

\bigskip
To prove the claim, we estimate $\|D_pF_x-D_qF_x\|$ in two directions, the one dimensional space $<X(x)>$ spanned by $X(x)$ and the normal space $N_x$.

Let $p=v_1+t_2X(x), ~ q=v_2+t_2X(x)\in U_x(r_0\|X(x)\|)$, where $v_1,v_2\in N_x$. If $d(p,q)<\delta \|X(x)\|$, then
$$\|(D_pF_x-D_qF_x)(X(x))\|=\|X(F_x(p))-X(F_x(q))\|$$$$\leq L\|F_x(p)-F_x(q)\|<2Ld(p,q)<2L\delta\|X(x)\|.$$
Hence  $$\|(D_pF_x-D_qF_x)|_{< X(x)>}\|<2L\delta.$$  On the other hand, for any $u\in N_x$,
$$\|(D_pF_x-D_qF_x)(u)\|=\|D_{x+v_1}\varphi_{t_1}(u)-D_{x+v_2}\varphi_{t_2}(u)\|$$
$$\leq\|D_{x+v_1}\varphi_{t_1}(u)-D_{x+v_1}\varphi_{t_2}(u)\|+\|D_{x+v_1}\varphi_{t_2}(u)-D_{x+v_2}\varphi_{t_2}(u)\|$$
$$\leq\|Id-D_{\varphi_{t_1}(x+v_1)}\varphi_{t_2-t_1}\|\cdot\|D_{x+v_1}\varphi_{t_1}(u)\|+\|D_{x+v_1}\varphi_{t_2}(u)-D_{x+v_2}\varphi_{t_2}(u)\|$$
$$\leq|e^{L|t_2-t_1|}-1|\cdot e^{L|t_1|}\cdot\|u\|+\|D_{x+v_1}\varphi_{t_2}(u)-D_{x+v_2}\varphi_{t_2}(u)\|$$
$$\leq e^{1/5}L|t_2-t_1|\cdot\|u\|+\|D_{x+v_1}\varphi_{t_2}(u)-D_{x+v_2}\varphi_{t_2}(u)\|$$
$$<2L\delta\cdot\|u\|+\|D_{x+v_1}\varphi_{t_2}-D_{x+v_2}\varphi_{t_2}\|\cdot\|u\|.$$
Since $M$ is compact, there is $\delta_0>0$ such that for any $y_1,y_2\in M$ with $d(y_1,y_2)<\delta_0$ and any $t\in[-r_0,r_0]$, one has
$$\|D_{y_1}\varphi_{t}-D_{y_2}\varphi_{t}\|<\frac{\epsilon}{4}.$$
Then for any $\epsilon>0$ there is  $\delta>0$ such that if $d(p,q)<\delta\|X(x)\|$ then
$$\|(D_pF_x-D_qF_x)|_{< X(x)>}\|<\frac{\epsilon}{2}, \ \|(D_pF_x-D_qF_x)|_{N_x}\|<\frac{\epsilon}{2}.$$
Hence $$\|(D_pF_x-D_qF_x)\|<\epsilon.$$ This proves the claim.

Now we consider the uniform (in a relative sense) continuity of $DF_x^{-1}$. It is easy to see that
$$\|D_{F_x(p)}F_x^{-1}-D_{F_x(q)}F_x^{-1}\|=\|(D_pF_x)^{-1}-(D_qF_x)^{-1}\|$$$$\leq\|(D_pF_x)^{-1}\|\cdot\|D_pF_x-D_qF_x\|\cdot\|(D_qF_x)^{-1}\|$$
for any $p, ~q\in U_x(r_0\|X(x)\|)$.
From the claim and the fact $\|DF_x^{-1}\|$ is bounded by $2$, it follows that for any $\epsilon>0$, there is $\delta>0$ such that if for any $y_1,y_2\in F_x(U_x(r_0\|X(x)\|))$, if $d(y_1,y_2)<\delta\|X(x)\|$, then $\|D_{y_1}F_x^{-1}-D_{y_2}F_x^{-1}\|<\epsilon$. Thus $$\|D_{y_1}P_x-D_{y_2}P_x\|<\epsilon.$$

Then for any $v,v'\in N_x(r_1\|X(x)\|)$,
$$\|D_vP_{x,T}-D_{v'}P_{x,T}\|=\|D_v(P_{\varphi_T(x)}\circ\varphi_T\circ\exp)-D_{v'}(P_{\varphi_T(x)}\circ\varphi_T\circ\exp)$$
$$\leq\|D_{\varphi_T(\exp_x(v))}P_{\varphi_T(x)}-D_{\varphi_T(\exp_x(v'))}P_{\varphi_T(x)}\|\cdot\|D_{\exp_x(v)}\varphi_T\|\cdot\|D_v\exp_x\|$$
$$+\|D_{\varphi_T(\exp_x(v'))}P_{\varphi_T(x)}\|\cdot\|D_{\exp_x(v)}\varphi_T-D_{\exp_x(v')}\varphi_T\|\cdot\|D_v\exp_x\|$$
$$+\|D_{\varphi_T(\exp_x(v'))}P_{\varphi_T(x)}\|\cdot\|D_{\exp_x(v')}\varphi_T\|\cdot\|D_v\exp_x-D_{v'}\exp_x\|.$$
By the uniform bound of $\|DP_x\|$ and $\|D\varphi_T\|$  and $\|D\exp_x\|$ and the continuity of $D_vP_x$ discussed above, it is straightforward to verify item 2. We omit the details.
\end{proof}

Given the rescaled shadowing condition $d(\varphi_t(x),\varphi_{\theta(t)}(y))\le\delta\|X(\varphi_t(x))\|$, there  is a time sequence $\{\theta(T_k)\}$ for $y$ that corresponds  to the time sequence $\{kT\}$ for $x$  as described in the next proposition.

\begin{Proposition}\label{sequence}
For any $T>r_0$, there is $\delta=\delta(T)>0$ with the following property: for any $x\in M\setminus Sing(X)$ and $y\in M$ and any increasing continuous function $\theta:\mathbb{R}\to \mathbb{R}$ with $d(\varphi_t(x),\varphi_{\theta(t)}(y))\le\delta\|X(\varphi_t(x))\|$ for any $t\in\mathbb{R}$, there is a sequence $\{T_k\}_{k\in\mathbb{Z}}$ such that
\begin{enumerate}
\item $\varphi_{\theta(T_k)}(y)\in\exp(N_{\varphi_{kT}(x)})$;
\item $\|\exp_{\varphi_{kT}(x)}^{-1}(\varphi_{\theta(T_k)}(y))\|\leq 3\delta\|X(\varphi_{kT}(x))\|;$
\item $|\theta(kT)-\theta(T_k)|\le 3\delta$;
\item $P_{\varphi_{kT}(x),T}(\exp_{\varphi_{kT}(x)}^{-1}(\varphi_{\theta(T_k)}(y)))=\exp_{\varphi_{(k+1)T}(x)}^{-1}(\varphi_{\theta(T_{k+1})}(y))$
\end{enumerate}
for any $k\in\mathbb{Z}$.
\end{Proposition}

Let us briefly explain the statement. If we think of $\exp_x(N_x(r))$, $r$ small, a local cross section at $x$ in $M$, then the flow $\varphi_t$ is transverse to the local sections and induces a holonomy map from $\exp_x(N_x(r))$ to
$\exp_{\varphi_t(x)}(N_{\varphi_t(x)})$. Item (1) just says that $\theta(T_k)$ is the time when the orbit of $y$ cuts the local section $\exp_{\varphi_{kT}(x)}(N_{\varphi_{kT}(x)})$ under the hololomy. Item (2) says the cut point is near the ``origin" $\varphi_{kT}(x)$. Item (4) just says these cuts are in the same orbit (of $y$).

\begin{proof}
Let $T>r_0$ be given. Choose $\epsilon>0$ such that
$$\epsilon T<r_0/2.$$ By Lemma \ref{lemma23}
there is $$\delta<r_0/12$$ such that for any $x\in
M\setminus {\rm Sing}(X)$, any $y\in M$, and any increasing continuous
function $\theta:[0,T]\to\mathbb{R}$, if
$$d(\varphi_t(x),\varphi_{\theta(t)}(y))\le\delta \|X(\varphi_t(x))\|$$ for
all $t\in[0,T]$, then $$|\theta(T)-\theta(0)-T|\le\epsilon T.$$
We also require
$$\delta<r_1/3=r_1(T)/3.$$

Assume we are given $x\in M\setminus {\rm Sing}(X)$ and $y\in M$ with a  increasing
continuous function $\theta:\mathbb{R}\to \mathbb{R}$ such that
$$d(\varphi_t(x),\varphi_{\theta(t)}(y))\le\delta\|X(\varphi_t(x))\|$$ for
all $t\in\mathbb{R}$.  We will look at the sequence of points $\varphi_{kT}(x)$ on ${\rm Orb}(x)$. For each $k\in \mathbb Z$, we will consider a flowbox of relative size $r_0$  around $\varphi_{kT}(x)$. The ``shadowing point" $\varphi_{\theta(kT)}(y)$ is in the flowbox and is near the center $\varphi_{kT}(x)$, but generally not in the normal direction of $X(\varphi_{kT}(x))$. Thus we need to consider a third point, to be denoted $\varphi_{\theta(T_k)}(y)$, the projection of $\varphi_{\theta(kT)}(y)$ along ${\rm Orb}(y)$ to the  normal direction of $X(\varphi_{kT}(x))$.
That is, on ${\rm Orb}(x)$ we will consider the point of time $kT$, while on ${\rm Orb}(y)$ we will consider the two points of time $\theta(kT)$ and $\theta(T_k)$, respectively.

Precisely, let $F_{\varphi_{kT}(x)}$ be the flowbox map of relative size $r_0$ at $\varphi_{kT}(x)$. Write $$F_{\varphi_{kT}(x)}^{-1}(\varphi_{\theta(kT)}(y))=u_k+t_k\cdot X(\varphi_{kT}(x)),$$
 where $u_k\in N_{\varphi_{kT}(x)}$, $t_k\in \mathbb{R}$. Then by the definition of the flowbox map, $$\varphi_{t_k}(\exp_{\varphi_{kT}(x)}(u_k))=\varphi_{\theta(kT)}(y).$$
Since $\|DF_x^{-1}\|\le3$ and $$d(\varphi_{kT}(x), \varphi_{\theta(kT)}(y))\le\delta\|X(\varphi_{kT}(x))\|,$$ we have $\|u_k\|\le3\delta\|X(\varphi_{kT}(x))\|$ and $|t_k|\le3\delta$.

By Corollary \ref{cor}, we may assume that $\theta$ is surjective. Then there is  $T_k$ such that
$$\theta(T_k)=\theta(kT)-t_k.$$
We prove the sequence $\{T_k\}_{k\in\mathbb{Z}}$ satisfies Proposition \ref{sequence}.

Since
$$\varphi_{\theta(T_k)}(y)=\varphi_{\theta(kT)-t_k}(y)=\exp_{\varphi_{kT}(x)}(u_k)\in \exp(N_{\varphi_{kT}(x)}),$$
item (1) holds. Since $$\|\exp_{\varphi_{kT}(x)}^{-1}(\varphi_{\theta(kT)-t_k}(y))=\|u_k\| \le3\delta\|X(\varphi_{kT}(x))\|,$$ item (2) holds. Since $$|\theta(kT)-\theta(T_k)|=|t_k|\le3\delta,$$ item (3) holds.

It remains to prove item (4), which is equivalent to
$$P_{\varphi_{kT}(x),T}(u_k)=u_{k+1}. $$
Note that $\|u_k\|\le3\delta\|X(\varphi_{kT}(x))\|$ and $\delta<r_1/3$, hence the sectional Poincar\'e map is well defined at $u_k$. By the definition of  $P_{\varphi_{kT}(x),T}$, this is the same as
$$P_{\varphi_{(k+1)T}(x)}(\varphi_T(\exp_{\varphi_{kT}(x)}(u_k)))=u_{k+1}.$$
Thus it suffices to find $s\in[-r_0,r_0]$ such that
$$\varphi_T(\exp_{\varphi_{kT}(x)}(u_k))=F_{\varphi_{(k+1)T}(x)}(u_{k+1}+s X(\varphi_{(k+1)T}(x) )),$$
or
$$\varphi_T(\exp_{\varphi_{kT}(x)}(u_k))=\varphi_{s}(\exp_{\varphi_{(k+1)T}(x)}(u_{k+1})).$$
This is the same as
$$\varphi_{T+\theta(kT)-t_k}(y)=\varphi_{s+\theta((k+1)T)-t_{k+1}}(y).$$
Let $$s=T+\theta(kT)-t_k-\theta((k+1)T)+t_{k+1}.$$ Then $$|s|\leq |\theta(kT)+T-\theta((k+1)T)|+|t_k|+|t_{k+1}|$$ $$\le\epsilon T+3\delta+3\delta\le r_0/2+6\delta\le r_0.$$
This proves item (4) and hence Proposition \ref{sequence}.

\end{proof}

\section{Proof of Theorems A}

We will reduce the problem of expansiveness  to the following Proposition \ref{expansive} of the problem of uniqueness of fixed points.

 For any $i\in \mathbb Z$, let $E_i$ be a
$d$-dimensional Euclidean space. Let $Y_0=\Pi_{i=-\infty}^\infty E_i$. For any $v=\{v_i\}\in Y_0$, denote
$\|v\|=\sup\{\|v_i\|\}$. Let
$$Y=\{v\in Y_0:\|v\|<+\infty\}.$$ Then $Y$ is a Banach space with norm $\|\cdot\|$.  For any $i\in \mathbb Z$, let
$$G_i:E_i\to E_{i+1}$$
be a map. These maps define a map  $G :Y\to Y$ by $$(G
v)_{i+1}=G_i(v_i).$$ In other words, $G$ is defined to be ``fiber-preserving" with respect to the shift map $i\to i+1$. Below in Proposition \ref{expansive} and Theorem A the map $G$ will be defined this way.

For any $i\in \mathbb Z$, assume
$E_i$ has a direct sum decomposition
$$E_i=\Delta^s_i\oplus
\Delta^u_i.$$ Define the {\it angle} between $\Delta^s_i$ and $\Delta^u_i$ by
$$
\angle(\Delta^s_i, \Delta^u_i)$$
$$=\inf\{\|u-v\|:(u\in \Delta^s_i, v\in \Delta^u_i, \|u\|=1) {\rm\ or\ }
(u\in \Delta^s_i, v\in \Delta^u_i, \|v\|=1)\}.
$$

\begin{Proposition}\label{expansive}
Let $\eta\in(0,1)$ and $\alpha>0$ be given. There is $\xi>0$
such that if  for every $i\in \mathbb Z$ the splitting $E_i=\Delta^s_i\oplus \Delta^u_i $ has angle
$\angle(\Delta^s_i, \Delta^u_i)>\alpha$, and if
$G:Y\to Y$ has the form $G_i=L_i+\phi_i:E_i\to E_{i+1}$, where $L_i$
is a linear isomorphism of the block form $ L_i=\left(\begin{array}{ll}
          A_i&0\\
          0&D_i
          \end{array}\right)
$ with respect to the splittings of $E_i$ and $E_{i+1}$ such that $\|A_i\|\leq\eta,\|D_i^{-1}\|\leq\eta$, and ${\rm
Lip}\phi_i<\xi$ and $\phi_i(0)=0$, then for any $v\in Y$, $G(v)=v$ implies $v=0$.
\end{Proposition}

Here $G$ is a Lipschitz perturbation of a ``hyperbolic" operator $L=\{L_i\}$. Since $\phi_i(0)=0$, we know $v=0$ is a fixed point of $G$ already. Then Proposition \ref{expansive} states that $v=0$ is the only fixed point of $G$. This is a classical result, see for instance Pilyugin [P] and Gan [G]. Since there is some slight difference here, for convenience we give the proof.

  \begin{proof} Let $L:Y\to Y$ and $\phi:Y\to Y$ denotes the maps defined by $(L(v))_{i+1}=L_i(v_i)$ and $(\phi(v))_{i+1}=\phi_i(v_i)$ respectively for any
$v=(v_i)$. Let $I$ be the identity map on $Y$. Then we know that
$I-L$ is invertible and $$(I-L)^{-1}=\left(\begin{array}{ll}
          (I-A)^{-1}&0\\
          0&(I-D)^{-1}
          \end{array}\right)$$
where $A=(A_i)$ and $D=(D_i)$. By the fact that $\|A\|\leq\eta$ and
$\|D^{-1}\|\leq\eta$ we know that
$$\|(I-A)^{-1}\|\leq\frac{1}{1-\eta}, ~~~~ \|(I-D)^{-1}\|\leq\frac{\eta}{1-\eta}.$$ For any $v=v_s+v_u$
with $\|v\|=1$, where $v_s\in \Delta^s$ and $v_u\in\Delta^u$, we have $$\|\frac{v_s}{\|v_s\|}+\frac{v_u}{\|v_s\|}\|\geq \alpha$$ by the
definition of angle. Then  $\|v_s\|\leq 1/\alpha$.
Similarly,  $\|v_u\|\leq 1/\alpha$. Hence
$$\|(I-L)^{-1}(v)\|=\|(I-A)^{-1}v_s+(I-D)^{-1}v_u\|$$ $$ \leq
\frac{1}{1-\eta}\|v_s\|+\frac{\eta}{1-\eta}\|v_u\|\leq\frac{1+\eta}{\alpha(1-\eta)}.$$
Hence we have $$\|(I-L)^{-1}\|\leq \frac{1+\eta}{\alpha(1-\eta)}.$$

It is easy to see that $G(v)=v$ is equivalent to $Lv+\phi(v)=v$ and
equivalent to $v=(I-L)^{-1}\phi (v)$. Now we consider a map
$T:Y\to
Y$ defined by
$$T(v)=(I-L)^{-1}\phi (v).$$ Then $T$ and $G$ have the same set of fixed points. For any $u,u'\in Y$, we have
$$\|T(u)-T(u')\|=\|(I-L)^{-1}(\phi(u)-\phi(u'))\|$$ $$\leq \frac{1+\eta}{\alpha(1-\eta)}\cdot{\rm Lip}\phi\cdot\|u-u'\|.$$
Now we choose $\xi>0$ such that
$$\frac{1+\eta}{\alpha(1-\eta)}\xi<1,$$ then $T$ is a contraction
mapping on $Y$ under the assumption ${\rm Lip}\phi<\xi$. We know that $T$ has a unique fixed point $0\in Y$ and  so does $G$. In other words, if there is $v\in Y$ such that $G(v)=v$, then $v=0$. This ends
the proof of Proposition \ref{expansive}.
\end{proof}

Now we prove theorems A, that is, every  multisingular
hyperbolic set is rescaling expansive. In the proof we will not assume the full strength of multisingular hyperbolicity but only a naive version of it. Anyway let us give it a name and a definition. Let $\Lambda$ be a compact invariant set of $X$. We will call a function
$h: (\Lambda\setminus{\rm Sing}(X))\times\mathbb{R}\to (0,+\infty)$  a  {\it naive cocycle} of $X$ on $\Lambda\setminus {\rm Sing}(X)$ if the following two conditions are satisfied:

(1) for any  $x\in\Lambda\setminus {\rm Sing}(X)$ and any $s, t\in \mathbb R$,
$h(x, s+t)=h(x, s)\cdot h(\varphi_s(x), t),$

(2) for any $x\in\Lambda\setminus {\rm Sing}(X)$, there is $K=K(x)>0$ such that
 $h(x, t)\le K$ for all $t\in \mathbb R$.

Following Bonatti-da Luz \cite{BL1} we will call a compact invariant set  $\Lambda$ of $X$ a
{\it naive multisingular hyperbolic set} of $X$ if, for some $C>1$ and $\lambda>0$, there is a $\psi_t$-invariant splitting $N_{\Lambda\setminus {\rm Sing}(X)}=\Delta^s\oplus \Delta^u$ such that

(1) $\Delta^s\oplus \Delta^u$ is a $(C,\lambda)$-dominated splitting with respect to $\psi_t$;

(2) there is a naive cocycle $h_t^s$ of $X$ such that $\Delta^s$ is $(C,\lambda)$-contracting for $h_t^s\cdot \psi_t$;

(3) there is a naive cocycle $h_t^u$ of $X$ such that $\Delta^u$ is $(C,\lambda)$-expanding for $h_t^u\cdot \psi_t$.

Note that this definition does not care about singularities and uses the usual linear Poincar\'e flow defined on $\Lambda\setminus {\rm Sing}(X)$. Let $h(e, t)$ be a pragmatical cocycle on $\tilde\Lambda$ with respect to a singularity $\sigma$ with isolating neighborhood $U$. It gives a cocycle $h(x, t):\Lambda\setminus {\rm Sing}(X)\times \mathbb R\to \mathbb R$ by $h(j(e), t)=h(e, t)$ for $e\in j^{-1}(\Lambda\setminus {\rm Sing}(X)).$ It is not hard to see that
$$h(x, t)\le K(x):=\max\{\frac{\sup_{x\in M}\|X(x)\|}{\|X(x)\|}, ~ \frac{\sup_{x\in M}\|X(x)\|}{\inf_{x\in \partial U}\|X(x)\|}\}$$
for any $x\in \Lambda\setminus {\rm Sing}(X)$ and $t\in \mathbb R$. Thus $h(x, t)$ is a naive cocycle. Hence a reparametrizing cocyle gives automatically a naive cocycle. Then one can easily check that every multisingular hyperbolic set is naive multisingular hyperbolic.

\bigskip

{\noindent\bf Proof of Theorem A}. ~
In fact we prove every naive multisingular
hyperbolic set is rescaling expansive. Let $\Lambda$ be a naive multisingular hyperbolic set of $X$ with a $(C, \lambda)$-dominated splitting
$$N_{\Lambda\setminus Sing(X)}=\Delta^s\oplus \Delta^u.$$ Let $L>0$ be a local Lipschitz constant of $X$.  Choose $T>0$ big enough such that $$\eta^2=Ce^{-\lambda T}<1.$$
Since $N_{\Lambda\setminus {\rm Sing}(X)}\Delta^s\oplus
\Delta^u$ is a dominated splitting, there is $\alpha>0$ such
that
$$\angle(\Delta^s(x),\Delta^u(x))>\alpha$$ for every $x\in\Lambda\setminus {\rm Sing}(X)$. Note that
this is guaranteed by the (uniform) dominance of the splitting on $\Lambda\setminus {\rm Sing}(X)$, even though $\Lambda\setminus {\rm Sing}(X)$ is non-compact. See \cite{HW} for
a proof.

Now we determine the number $\epsilon_0>0$ for Theorem A, which is supposed to depend only on the vector field $X$ and the set $\Lambda$ (and hence on $L$, $T$, $\eta$, and $\alpha$) but not on $x$, $y$  and others.

Let $\xi>0$ be the constant given in Proposition
\ref{expansive} associated to $\eta$ and $\alpha$. Take $\epsilon_0>0$ so that
$$\epsilon_0\le \min\{r_1/3, ~ 3\delta(T)\},$$
where $r_1=r_1(T)$ is the number in the definition of the sectional Poincar\'e map, and $\delta(T)$ is the number in the statement of Proposition 3.2. Also, by item 2 of Proposition \ref{prop3.1}, we can  take $\epsilon_0$ so that for any regular point $z$ of $X$, if
$z'\in\exp_z(N_z)$ and $d(z',z)<3\epsilon_0\|X(z)\|$ then
$$\|D_{\exp^{-1}_z(z')}P_{z,T}-D_0P_{z,T}\|<\frac{\xi}{5\eta^{-1}e^{LT}}.$$
Here $\xi/(5\eta^{-1}e^{LT})$ and $3\epsilon_0$ play the role of ``$\epsilon$" and ``$\delta$"
in the statement of Proposition \ref{prop3.1}. Since $D_0P_{z,T}=\psi_T|_{N_z}$, this is the same as $$\|D_{\exp^{-1}_z(z')}P_{z,T}-\psi_T|_{N_z}\|<\frac{\xi}{5\eta^{-1}e^{LT}}.$$
This settles the choice of $\epsilon_0>0$.

Let $h_t^s$ and $h_t^u$ be two naive cocycles such that $h_t^s\cdot\psi_t|_{\Delta^s}$ is $(C,\lambda)$-contracting and $h_t^u\cdot\psi_t|_{\Delta^u}$ is $(C,\lambda)$-expanding. Then for any
$x\in\Lambda\setminus {\rm Sing}(X)$,

(a) ~ $\|\psi_T|_{\Delta^s_x}\|\cdot\|\psi_{-T}|_{\Delta^u(\varphi_T(x))}\|\le\eta^2$;

(b) ~ $h_T^s(x)\cdot\|\psi_T|_{\Delta^s(x)}\|\le\eta$;

(c) ~ $h_T^u(x)\cdot m(\psi_T|_{\Delta^u(x)})\ge\eta^{-1}$.

The key to the proof of Theorem A is the following
\vskip 0.2cm
{\noindent\bf Claim.}  {\it For every $x\in \Lambda\setminus {\rm Sing}(X)$, there is a sequence $\{c_i=c_i(x)>0: i\in \mathbb{Z}\}$ such that the following three conditions hold:

$(A1)$ The set $\{c_i(x): x\in \Lambda\setminus {\rm Sing}(X),~i\in \mathbb Z \}$ of numbers is bounded.

$(A2)$ For every $x\in \Lambda\setminus {\rm Sing}(X)$, $c_i\cdot\|\psi_T|_{\Delta^s(\varphi_{iT}(x))}\|\le\eta$, and $
c_i\cdot m(\psi_T|_{\Delta^u(\varphi_{iT}(x)})\ge\eta^{-1}.$

$(A3)$ Denote $b_i=b_i(x)=c_0\cdot c_1\cdots c_{i-1}$ for $i>0$ and  $b_i=b_i(x)=c_{i}^{-1}\cdot c_{i+1}^{-1}\cdots c_{-1}^{-1}$ for $i<0$. Then for every $x\in \Lambda\setminus {\rm Sing}(X)$, the sequence $\{b_i(x)\}_{i\in Z}$  is bounded.}

\vskip 0.2cm Briefly, condition (A2) says that, replacing $h_T^s(\varphi_{iT}(x))$ and $h_T^u(\varphi_{iT}(x))$ both by $c_i$,  items  (b) and (c) hold simultaneously. Condition (A3) says a ``bounded product" property.

\vskip 0.2cm

\nt {\it Proof of the Claim}. Let $x\in\Lambda\setminus {\rm Sing}(X)$. We define $c_i$ by two different formulas depending on $i\ge 0$ or $i<0$. If $i\ge 0$, let $$c_i=c_i(x)=\frac{\eta^{-1}}{m(\psi_T|_{\Delta^u(\varphi_{iT}(x))})}.$$ Then

(1) ~ $\eta^{-1}e^{-LT}\leq c_i\leq \eta^{-1} e^{LT}$;

(2) ~ $c_i\cdot m(\psi_T|_{\Delta^u(\varphi_{iT}(x))})=\eta^{-1}$;

(3) ~ $c_i\cdot \|\psi_T|_{\Delta^s(\varphi_{iT}(x))}\|\le\eta$;

(4) ~ $c_i\le h_T^u(\varphi_{iT}(x))$.

Thus (1) verifies condition (A1) for $i\ge 0$, and (2) and (3) verify  condition (A2) for $i\ge 0$.

If $i<0$,  let $$c_i=c_i(x)=\frac{\eta}{\|\psi_T|_{\Delta^s(\varphi_{iT}(x))}\|}.$$ Then

(1*) ~ $\eta e^{-LT}\leq c_i\leq \eta e^{LT}$;

(2*) ~ $c_i\cdot \|\psi_T|_{\Delta^s(\varphi_{iT}(x))}\|=\eta$;

(3*) ~ $c_i\cdot m(\psi_T|_{\Delta^u(\varphi_{iT}(x))})\ge\eta^{-1}$;

(4*) ~ $c_i\ge h_T^s(\varphi_{iT}(x))$.

Thus (1*) verifies condition (A1) for $i< 0$, and (2*) and (3*) verify condition (A2) for $i< 0$.

Let $b_0=1$. For every $i>0$,  by (4),
$$b_i=b_i(x)=c_0\cdot c_1\cdots c_{i-1}$$
$$\le h_T^u(x)\cdot h_T^u(\varphi_T(x))\cdots\ h_T^u(\varphi_{(i-1)T}(x))=h_{iT}^u(x). $$
For every $i<0$, by (4*),
 $$b_i=b_i(x)=c_{i}^{-1}\cdot c_{i+1}^{-1}\cdots c_{-1}^{-1}$$
 $$\le[h_T^s(\varphi_{iT}(x))\cdot h_T^s(\varphi_{(i+1)T}(x))\cdots\ h_T^s(\varphi_{-T}(x))]^{-1}$$ $$=[h^s_{-iT}(\varphi_{iT}(x))]^{-1}=h^s_{iT}(x).$$
 By the definition of naive cocycle,  for fixed $x$, the two sequences $\{h_{iT}^u(x)\}_{i\in \mathbb Z}$ and $\{h^s_{iT}(x)\}_{i\in \mathbb Z}$ are bounded. Thus the sequence $\{b_i(x)\}_{i\in \mathbb Z}$ is bounded. This verifies condition (A3), proving the Claim.

Now let $0<\epsilon\leq\epsilon_0$. Let $x\in \Lambda$ and $y\in M$ and an
 increasing continuous function $\theta:\mathbb{R}\to\mathbb{R}$ be given such that
$$d(\varphi_t(x),\varphi_{\theta(t)}(y))\le(\epsilon/3)\|X(\varphi_t(x))\|$$ for all $t\in\mathbb{R}$.
From now on  $x$ and $y$ will be fixed till the end of the proof of Theorem A. We prove $$\varphi_{\theta(t)}(y)\in \varphi_{[-\epsilon,~ \epsilon]}(\varphi_t(x))$$ for all $t\in \mathbb R$.
We assume $x\notin {\rm Sing}(X)$ because otherwise the situation would be trivial.

Denote
$$E_i=N_{\varphi_{iT}(x)}.$$

Let $\beta:[0,+\infty)\to[0,1]$ be a bump function such that

(a) ~ $\beta(t)=1$ for $t\in[0,1/3]$;

(b) ~ $\beta(t)=0$ for $t\in[2/3,+\infty)$;

(c) ~ $\beta'(t)\in[-4,0]$ for any $t\in [0,+\infty)$.

Define
$$P_i:E_i\to
E_{i+1}$$ to be
$$P_i(v)=\beta(\frac{\|v\|}{3\epsilon\|X(\varphi_{iT}(x))\|})\cdot P_{\varphi_{iT}(x),T}(v)+(1-\beta(\frac{\|v\|}{3\epsilon\|X(\varphi_{iT}(x))\|}))\cdot \psi_T(v).$$
 Roughly, we use the bump function $\beta$ to extend the local map $P_{\varphi_{iT}(x),T}$ defined near the origin of $E_i$  to the whole $E_i$ so that it agrees with $\psi_T$ away from the origin. Precisely,
$P_i=P_{\varphi_{iT}(x), T}$
 inside the ball $N_{\varphi_{iT}(x)}(\epsilon\|X(\varphi_{iT}(x))\|)$, and $P_i=\psi_T$ outside the ball  $N_{\varphi_{iT}(x)}(3\epsilon\|X(\varphi_{iT}(x))\|)$. Note that $3\epsilon\le r_1$, hence $P_{\varphi_{iT}(x), T}$ is well defined in the ball $N_{\varphi_{iT}(x)}(3\epsilon\|X(\varphi_{iT}(x))\|)$, and hence $P_i$ is well defined on the whole $E_i$.  A direct computation gives
$$\|D_vP_i-\psi_T|_{E_i}\|<\frac{\xi}{\eta^{-1}e^{LT}}, ~~ \forall ~ v\in E_i.$$

\bigskip
\nt {\bf Remark.} For convenience we sketch the computation.
Abbreviate $r=3\epsilon\|X(\varphi_{iT}(x))\|$. Then
$$P_i(v)-\psi_T(v)=\beta(\frac{\|v\|}{r})\cdot (P_{\varphi_{iT}(x),T}(v)-\psi_T(v)).$$
We may assume $\|v\|\le r$ because otherwise the value is 0. Hence
$$D_vP_i-\psi_T$$ $$=\beta'(\frac{\|v\|}{r})\cdot \frac{1}{r}\cdot\frac{v}{\|v\|}\cdot (P_{\varphi_{iT}(x),T}(v)-\psi_T(v))+\beta(\frac{\|v\|}{r})\cdot(D_vP_{\varphi_{iT}(x),T}-\psi_T).$$
By the generalized mean value theorem,
$$\|P_{\varphi_{iT}(x),T}(v)-\psi_T(v)\|\leq \frac{\xi}{5\eta^{-1}e^{LT}}\|v\|.$$
Thus
$$\|D_vP_i-\psi_T\|\leq |\beta'(\frac{\|v\|}{r})|\cdot \frac{\xi}{5\eta^{-1}e^{LT}}\cdot\frac{\|v\|}{r}+|\beta(\frac{\|v\|}{r})|\cdot\frac{\xi}{5\eta^{-1}e^{LT}}\leq \frac{\xi}{\eta^{-1}e^{LT}}.$$
The last step uses the facts that $|\beta'|\le 4$, $\|v\|\le r$, and $|\beta|\le 1$. This ends the remark.

\bigskip
Since $\epsilon/3\le \epsilon_0/3\le \delta(T)$, by Proposition \ref{sequence}, there is a sequence $\{T_i:i\in\mathbb{Z}\}$ such that $\varphi_{\theta(T_i)}(y)\in\exp(N_{\varphi_{iT}(x)}).$
Let $$u_i=\exp_{\varphi_{iT}(x)}^{-1}(\varphi_{\theta(T_i)}(y))\in E_i.$$ By items 2 and 4 of Proposition \ref{sequence}, we have $$\|u_i\|\le \epsilon\|X(\varphi_{iT}(x))\|, ~~ P_{\varphi_{iT}(x),T}(u_i)=u_{i+1}.$$ That is,
$$P_i(u_i)=u_{i+1}.$$
Let
$u=(u_i)_{i\in \mathbb{Z}}.$ Since $M$ is compact, $\{\|X(z)\|\}_{z\in M}$ is bounded. Hence
$$\|u\|=\sup\{\|u_i\|:i\in\mathbb{Z}\}<+\infty,$$ i.e., $u\in Y.$ Here $Y$ consists of all bounded elements of $Y_0$, where $Y_0=\Pi_{i=-\infty}^\infty E_i$ (see the beginning of section 4 for notations).

Now define
$$G_i: E_i\to E_{i+1}$$
to be
$$G_i(v)= b_{i+1} P_i (b_i^{-1} v).$$ Here $b_i=b_i(x)$ is given by the Claim, where $x$ is the point that has been fixed such that $0_x$ is the origin of $E_0$. Since $$b_{i+1}\cdot b_i^{-1}=c_i\in[\eta e^{-LT}, \eta^{-1}e^{LT}]$$ (condition (A1)), and $$\|D_vP_i-\psi_T|_{E_i}\|<\frac{\xi}{\eta^{-1}e^{LT}}, ~~ \forall ~ v\in E_i,$$   we have $$\|D_vG_i-c_i\psi_T|_{E_i}\|<\xi, ~~ \forall ~ v\in E_i.$$ Define $G:Y\to Y_0$ to be
$$G|_{E_i}=G_i.$$  Since $G(0)=0,$
the derivative condition $\|D_vG_i-c_i\psi_T|_{E_i}\|<\xi$ guarantees that $G$ maps a bounded element of $Y_0$ to a bounded element of $Y_0$. That is, $G$ maps $Y$ into $Y$ and hence
$$G:Y\to Y$$ is well defined. Write
$$G_i=c_i\psi_T|_{E_i}+\phi_i.$$ Then
$${\rm Lip}(\phi_i)<\xi.$$
By condition (A2), $c_i\psi_T|_{E_i}$ can serve as the operator $L_i$ of Proposition \ref{expansive}.

Let $w_i=b_i u_i$. Then  $$G_i(w_i)=b_{i+1}P_i(b_i^{-1}b_i u_i)=b_{i+1}u_{i+1}=w_{i+1}.$$ That is,
$$G(w)=w,$$ where $w=(w_i)_{i\in \mathbb{Z}}.$ By condition (A3), the sequence $\{b_i\}_{i\in \mathbb Z}$ is bounded.
Then $$\|w\|=\sup\{\|w_i\|:i\in\mathbb{Z}\}<+\infty,$$ i.e.,  $w\in Y.$ Therefore, by Proposition \ref{expansive}, $w=0$. From $w_i=0$ we get $v_i=0.$ That is, $$\varphi_{\theta(T_i)}(y)=\varphi_{iT}(x).$$ By Proposition \ref{sequence}, $$|\theta(T_i)-\theta(iT)|\le3\cdot (\epsilon/3)=\epsilon.$$ Then  $$\varphi_{\theta(iT)}(y)=\varphi_{\theta(iT)-\theta(T_i)}(\varphi_{\theta(T_i)}(y))$$ $$=\varphi_{\theta(iT)-\theta(T_i)}(\varphi_{iT}(x))\in\varphi_{[-\epsilon, ~\epsilon]}(\varphi_{iT}(x)).$$

Now for any $\tau\in\mathbb{R}$,  set
$$z=\varphi_\tau(x), ~ y_1=\varphi_{\theta(\tau)}(y), ~ \theta_1(t)=\theta(t+\tau)-\theta(\tau).$$ Then
$$d(\varphi_{\theta_1(t)}(y_1),\varphi_t(z))=d(\varphi_{\theta(t+\tau)}(y),\varphi_{t+\tau}(x))$$ $$\le(\epsilon/3)\|X(\varphi_{t+\tau}(x))\|=(\epsilon/3)\|X(\varphi_t(z))\|.$$ Hence $y_1=\varphi_{\theta_1(0)}(y_1)\in\varphi_{[-\epsilon, ~\epsilon]}(z)$. Thus $$\varphi_{\theta(\tau)}(y)\in\varphi_{[-\epsilon, ~\epsilon]}(\varphi_{\tau}(x)).$$
This ends the proof of Theorem A.

\begin{Corollary}\label{Corollary}
Let $\Lambda$ be a singular
hyperbolic set of a $C^1$ vector field $X$ on $M$. Then $\Lambda$
is rescaling expansive. In fact, there is $\epsilon_0>0$ such that for any $0<\epsilon\leq\epsilon_0$, any $x\in\Lambda$ and $y\in M$, and any increasing continuous functions $\theta: \mathbb R\to \mathbb R$, if
$d(\varphi_{\theta(t)}(y),\varphi_t(x))\le(\epsilon/3)\|X(\varphi_t(x))\|$ for all
$t\in\mathbb{R}$, then  $\varphi_{\theta(t)}(y)\in \varphi_{[-\epsilon, \epsilon]}(\varphi_t(x))$ for all $t\in \mathbb R$.

\end{Corollary}

The next proposition explains why this is a corollary of Theorem A.

\begin{Proposition}\label{lemma32}
Every  singular hyperbolic set  is multisingular hyperbolic.

\end{Proposition}
\begin{proof}

Let $\Lambda$ be a $(C, \lambda)$-singular hyperbolic set of $X$ with
dominated splitting $E\oplus F$ of $\Phi_t$. Without loss of generality we assume $\Lambda$ is positive singular hyperbolic for $X$.  Thus $E$ is $(C, \lambda)$-contracting and $F$ is $(C, \lambda)$-area-expanding with respect to $\Phi_t$. First we work on  the usual linear Poincar\'e flow on $\Lambda\setminus
{\rm Sing}(X)$. For $x\in \Lambda\setminus
{\rm Sing}(X)$, let $h^s_t(x)\equiv 1$ and $h^u_t(x)=\|\Phi_t|_{\langle X(x)\rangle}\|$.
Note that since
$\|\Phi_t|_{\langle X(x)\rangle}\|={\|X(\varphi_t(x))\|}/{\|X(x)\|},$ $h_t^u$ satisfies the  cocycle condition. We prove  the following three items:
\begin{enumerate}
\item there is a dominated splitting $N_{\Lambda\setminus {\rm Sing}(X)}=\Delta^s\oplus \Delta^u$ with respect to $\psi_t$;
\item $\Delta^s$ is contracting for $\psi_t$;
\item $\Delta^u$ is expanding for $\|\Phi_t|_{<X(x)>}\|\cdot\psi_t$.
\end{enumerate}

The proof is straightforward. First note that,  for any $x\in \Lambda\setminus {\rm Sing}(X)$, the flow direction
$X(x)$ is contained in the linear subspace $F_x$ of $T_x M$. In fact, if $X(x)\notin F_x$ then, by dominance, $\Phi_{-t}(X(x))$ would accumulate on $E$ when $t\to +\infty$, hence
$$\|X(\varphi_{-t}(x))\|=\|\Phi_{-t}(X(x))\|$$  would grow exponentially, contradicting
that $\|X(x)\|$ is bounded above on $\Lambda$. Thus $X(x)\in F_x$ for any $x\in \Lambda\setminus
{\rm Sing}(X)$.

Now we proceed to find $\Delta^s$ and $\Delta^u$ in
$N_{\Lambda\setminus {\rm Sing}(X)}$. For any $x\in\Lambda\setminus {\rm Sing}(X)$, let $\pi_x: T_x M\to N_x$ be the
orthogonal projection. Put $$\Delta^s(x)=\pi_x(E_x).$$ Since the angle between $E$ and $F$ is positive and since $\langle X \rangle\subset F$, the angles between $E(x)$ and
$X(x)$ for all $x\in\Lambda\setminus {\rm Sing}(X)$ have a positive lower bound. In particular, $\pi_x|_{E_x}$ is a linear isomorphism and ${\rm dim}\Delta^s(x)={\rm dim}E_x$. Also, there is $K>1$ such that
 $$m(\pi_x|_{E_x})>K^{-1}$$ for all $x\in\Lambda\setminus {\rm Sing}(X)$. By the
invariance of subbundle $E$ and the definition of $\psi_t$ we can
easily check that $$\psi_t(\Delta^s(x))=\Delta^s(\varphi_t(x)).$$
Then for any unit vector $u\in \Delta^s_x$,
$$\|\psi_t(u)\|=\|\pi_{\varphi_t(x)}\circ\Phi_t\circ(\pi_x|_{E_x})^{-1}(u)\|<KCe^{-\lambda
t},$$ proving (2).

Then we put $$\Delta^u(x)=\pi_x(F_x)=N_x\cap F_x.$$ For any unit
vector $v\in \Delta^u(x)$, let  $L$ be the plane spanned by
$v$ and $X(x)$. Then
$$\|\Phi_{-t}|_{<
X(x)>}\|\cdot \|\psi_{-t}(v)\|=|{\rm
det}(\Phi_{-t}|_{L})|<Ce^{-\lambda t}$$ for any $t>0$, proving (3).

 We verify that $\Delta^s\oplus \Delta^u$ is a dominated splitting with respect to $\psi_t$. There is a tricky point here as we have to look at the negative direction of the flow: For any unit vectors $u\in \Delta^s_x$ and $v\in \Delta^u_x$ and any $t>0$,
$$\frac{\|\psi_{-t}(v)\|}{\|\psi_{-t}(u)\|}\leq \frac{\|\Phi_{-t}(v)\|}{\|\pi_{\varphi_{-t}(x)}\Phi_{-t}(u')\|}\leq \frac{\|\Phi_{-t}(v)\|}{K^{-1}\|\Phi_{-t}(u')\|}$$
$$=\frac{K\|\Phi_{-t}(v)\|}{\|u'\|\|\Phi_{-t}(u'/{\|u'\|})\|}\leq K Ce^{-\lambda t},$$
 where $u'=(\pi_x|_{E_x})^{-1}(u)$. The last inequality uses the fact that $u'\in E_x$ and $v\in F_x$. This proves (1).

 Now we extend everything to $\tilde\Lambda$. Denote
 $$\Lambda^\#=\{X(x)/\|X(x)\|: x\in \Lambda\setminus {\rm Sing}(X)\}. $$
 Then $\overline{\Lambda^\#}=\tilde\Lambda$. The cocycle $$h:(\Lambda\setminus {\rm Sing}(X))\times \mathbb {R} \to (0, \infty)$$
 $$h(x, t)=\|\Phi_t|_{\langle X(x)\rangle}\|=\|\Phi_t(\frac{X(x)}{\|X(x)\|})\|$$
 gives a cycle $$h:\Lambda^\#\times \mathbb {R} \to (0, \infty)$$
 $$h(e, t)=\|\Phi_t(e)\|,$$
 which is uniformly continuous and hence extends to a (reparametrizing) cocycle
 $$\tilde h:\tilde\Lambda\times \mathbb {R} \to (0, \infty)$$
 $$\tilde h(e, t)=\|\Phi_t(e)\|.$$

 As usual, the dominated splitting $\Delta^s\oplus \Delta^u$ of $\psi_t$ extends to a dominated splitting (still denoted) $\Delta^s\oplus \Delta^u$ of $\tilde\psi_t$ such that items (1) through (3) still hold. This proves
 Proposition \ref{lemma32}.
\end{proof}

\section{Proof of Theorem B}

 For preciseness  we use  sometimes the notation $\varphi_t^X$ to denote the flow generated by the vector field $X$.
\begin{Lemma}\label{C1}
Let $X\in \mathcal{X}^1(M)$ and let $Q$ be a non-hyperbolic periodic orbit of $X$. For any neighborhood $\mathcal{U}$ of $X$, any neighborhood $U$ of $\gamma$ and any $\delta>0$, there is $Y\in\mathcal{U}$ such that:
\begin{enumerate}
\item $X=Y$ outside $U$;
\item there exist two distinct hyperbolic periodic orbits $Q_1$ and $Q_2$ of $Y$ contained in $U$ with a time reparametrization $\theta:\mathbb{R}\to\mathbb{R}$ such that for some $x\in Q_1$ and $x\in Q_2$, $d(\varphi^Y_t(x),\varphi^Y_{\theta(t)}(y))<\delta \|Y(\varphi_t(x))\|$ for all $t\in\mathbb{R}$.
\end{enumerate}
\end{Lemma}

\begin{proof}
Let $Q$ be a non-hyperbolic periodic orbit of $X$. Take $q\in Q$ and denote $N_q(r)$ the $r$-ball of the center the origin $0_q$ in the normal space $N_q$. There is  $r>0$ such that the first return map $f^X_q:N_p(r)\to N_p$ of $X$ is well defined such that $D_{0_p}f^X_q$ has an eigenvalue $\lambda$ on the unit circle. With an arbitrarily small perturbation if necessary, we can assume $\lambda$ is a simple eigenvalue and there are no other eigenvalues of $D_{0_p}f^X_q$ on the unit circle except $\lambda$ and $\bar{\lambda}$. Let $V$ be the eigenspace associated to $\lambda$. If $\lambda$ is complex, with an arbitrarily small perturbation near $Q$ if necessary, we can assume that $D_{0_p}f^X_q|_{V}$ is a rational rotation. In any case we can assume that $(D_{0_p}f^X_q)^k|_{V}=id$ for some positive integer $k$. Then by a standard perturbation argument (see Lemma 1.3 of \cite{MSS} for a precise proof), there is $Y_0$ arbitrarily close to $X$ that keeps the orbit $Q$ unchanged such that $f_p^{Y_0}=\exp_p\circ D_{0_q}f^X_q\circ\exp_q^{-1}$ in a small neighborhood of $0_p$ where $f_p^{Y_0}$ denotes the first return map of $Y_0$ on $N_q$. Note that all perturbations here can be supported on an arbitrarily small neighborhood of $Q$. That is, we can assume $Y_0=X$ on $M\setminus U$ for any given neighborhood $U$ of $Q$.

Thus we may assume that $U$ contains no singularities of $X$, and hence there is $a>0$ such that $\|X(x)\|>a$ for every $x\in U$. Since $Y_0$ can be chosen arbitrarily close to $X$, we can assume $\|Y_0(x)\|>a$ for all $x\in U$.
Since $(f_p^{Y_0})^k=id$ in a small disc or arc  in $\exp_p(V)$ centered at $0_p$, for any $\delta>0$, we can choose distinct $x,y\in \exp_p(V)$ arbitrarily close to $q$ together with an increasing homeomorphism $\theta:\mathbb{R}\to\mathbb{R}$ such that $d(\varphi_t^{Y_0}(x),\varphi_{\theta(t)}^{Y_0}(y)<a\delta$ for all $t\in\mathbb{R}$. With an arbitrarily small perturbation $Y$ of $Y_0$ that keeps the orbits of $x,y$ unchanged, we may assume $Q_1={\rm Orb}(x)$ and $Q_2={\rm Orb}(y)$ are hyperbolic.  This ends the proof of  Lemma \ref{C1}.
\end{proof}

\begin{Proposition}\label{R_1}
There is a residual set $\mathcal{R}_1\subset \mathcal{X}^1(M)$ such that, for any
$X\in\mathcal{R}_1$, if there are $X_n\to X$ and non-hyperbolic periodic orbits $Q_n$ of $X_n$ that converge to a compact set $\Gamma$ in the Hausdorff metric, then there are  two sequences of hyperbolic periodic points $\{p_n\},\{q_n\}$ of $X$ with the following properties:

(1) for any $n$, ${\rm Orb}(p_n)\neq {\rm Orb}(q_n)$, and there is an increasing homeomorphism $\theta_n:\mathbb{R}\to \mathbb{R}$
such that $d(\varphi_t(p_n),\varphi_{\theta_n(t)}(q_n))<({1}/{n}) \|X(\varphi_t(p_n))\|$ for all $t\in\mathbb{R}$.

(2) ${\rm Orb}(p_n)$ and ${\rm Orb}(q_n)$ converge to $\Gamma$ in the Hausdorff metric.
\end{Proposition}

\begin{proof}
Let $\mathcal{K}(M)$ be the space of
nonempty compact subsets of $M$ with the Hausdorff metric, and
 $\{ \mathcal{O}_n \}_{n=1}^{\infty}$ be a countable basis of
$\mathcal{K}(M)$. For each pair of positive integers $n$ and $k$, denote by $\mathcal{H}_{n,k}$ the subset of $\mathcal{X}^1(M)$ such that any $Y\in\mathcal{H}_{n,k}$ has a $C^1$ neighborhood $\mathcal{V}$ in $\mathcal{X}^1(M)$ such that every $Z\in\mathcal{V}$ has two hyperbolic periodic points $p$ and $q$ such that (a) ${\rm Orb}(p)$ and ${\rm Orb}(q)$ are distinct and both in $\mathcal{O}_k$, (b) there is an increasing homeomorphism $\theta:\mathbb{R}\to\mathbb{R}$ such that
$$d(\varphi^Z_t(p),\varphi^Z_{\theta(t)}(q))<\frac{1}{n}\|Z(\varphi_t^Z(p))\|$$
for all $t\in\mathbb{R}$.

Let $\mathcal{N}_{n,k}$ be the complement of the $C^1$-closure of $\mathcal{H}_{n,k}$. Clearly,  for every pair $(n,k)$, $\mathcal{H}_{n,k}\cup\mathcal{N}_{n,k}$ is $C^1$ open and dense in $\mathcal{X}^1(M)$. Let $\mathcal{KS}$ denote the set of Kupka-Smale systems in $\mathcal{X}^1(M)$. Denote
$$\mathcal{R}_1=(\bigcap_{n, k\in\mathbb{N}}(\mathcal{H}_{n, k}\cup\mathcal{N}_{n, k}))\cap \mathcal{KS}.$$
Then $\mathcal{R}_1$ is $C^1$ residual.

Let $X\in\mathcal{R}_1$. Assume $X_n\to X$. Also assume there are non-hyperbolic periodic orbits $Q_n$ of $X_n$ that converge to a compact set $\Gamma$ in Hausdorff metric. Then for any neighborhood $\mathcal{O}$ of $\Gamma$ in $\mathcal{K}(M)$, there is $\mathcal{O}_k$ with $\Gamma\in\mathcal{O}_k\subset \mathcal{O}$. By Lemma 5.1, for any positive integer $n$ and any neighborhood $\mathcal{U}$ of $X$, there are $Z\in \mathcal{U}$ and two  hyperbolic periodic points $p$ and $q$ of $Z$ such that (a) ${\rm Orb}(p)$ and ${\rm Orb}(q)$ are distinct and both  in $\mathcal{O}_k$, (b) there is an increasing homeomorphism $\theta:\mathbb{R}\to\mathbb{R}$ such that
$$d(\varphi^Z_t(p),\varphi^Z_{\theta(t)}(q))<\frac{1}{n}\|Z(\varphi_t^Z(p))\|$$
for all $t\in\mathbb{R}$. Since a hyperbolic periodic orbit is persistent under $C^1$ perturbations, $Z\in\mathcal{H}_{n,k}$. Hence for any pair of positive integers $(n,k)$, $X$ is in the closure of $\mathcal{H}_{n,k}$ and hence not in $\mathcal{N}_{n,k}$. This means $X\in\mathcal{H}_{n,k}$ for all $(n,k)$. Hence for any neighborhood $\mathcal{O}$ of $\Gamma$ and any positive integer $n$, $X$ has two hyperbolic periodic points $p$ and $q$ of distinct orbits that are in $\mathcal{O}$ together with an increasing homeomorphism $\theta:\mathbb{R}\to \mathbb{R}$
such that $d(\varphi_t(p),\varphi_{\theta(t)}(q))<(1/n) \|X(\varphi_t(p)\|$ for all $t\in\mathbb{R}$. This ends the proof of Proposition \ref{R_1}.
\end{proof}

We need the recent result of Bonatti-da Luz:

\begin{Proposition}\label{BL} $($\cite{BL1, BL2}$)$
There is a residual set $\mathcal{R}_2\subset \mathcal{X}^1(M)$ such that any $X\in\mathcal{R}_2$ is a star flow if and only if any chain class $\Lambda$ of $X$ is multisingular hyperbolic.
\end{Proposition}

Now we prove Theorem B.

\bigskip
{\noindent\bf Proof of Theorem B}. ~ Let
$$\mathcal{R}=\mathcal{R}_1\cap \mathcal{R}_2.$$
We prove $\mathcal{R}$ satisfies Theorem B. Thus let $X\in \mathcal{R}$ and let $\Lambda$ be an isolated chain transitive set of $X$. We prove the three items of Theorem B circularly. Since $(2)\Rightarrow (3)$ is guaranteed by Proposition \ref{BL} (Proposition \ref{BL} is global, but it obviously applies to our case of an isolated chain transitive set)  and $(3)\Rightarrow (1)$ is guaranteed by Theorem A, it remains to prove $(1)\Rightarrow (2)$.

\vskip 0.2cm

\nt {\it Proof.} Assume $\Lambda$ is rescaling expansive. We prove $\Lambda$ is locally star. Suppose for the contrary there are  $X_n\to X$ with non-hyperbolic periodic orbits $Q_n$ of $X_n$ that converge to a compact set $\Gamma\subset\Lambda$ in the Hausdorff metric. By Proposition \ref{R_1}, there are two sequences of hyperbolic periodic points $\{p_n\},\{q_n\}$ of $X$ with the following properties:

(1) for any $n$, ${\rm Orb}(p_n)\neq{\rm Orb}(q_n)$, and there is an increasing homeomorphism $\theta_n:\mathbb{R}\to \mathbb{R}$
such that $d(\varphi_t(p_n),\varphi_{\theta_n(t)}(q_n))<({1}/{n}) \|X(\varphi_t(p_n))\|$ for all $t\in\mathbb{R}$.

(2) ${\rm Orb}(p_n)$ and ${\rm Orb}(Q_n)$ converge to $\Gamma$ in the Hausdorff metric.

Since $\Lambda$ is isolated for $X$,  $p_n,q_n\in\Lambda$ for large $n$. This contradicts the assumption that $\Lambda$ is rescaling expansive. This proves $(1)\Rightarrow (2)$ and hence Theorem B.

\vskip 0.2cm
\nt {\bf Remark.} We may add item (4) as to be ``$\Lambda$ is naive multisingular hyperbolic for $X$". Then the four items are equivalent. This is  because $(3)\Rightarrow (4)$ is obvious and  $(4)\Rightarrow (1)$ is contained in (the proof of) Theorem A. In other words, generically, naive multisingular hyperbolicity is equivalent to multisingular hyperbolicity.

\section{Appendix}

In this appendix we discuss the equivalence between the rescaled expansiveness and the Komuro expansiveness for non-singular flows. We also include a third condition, the expansiveness of Bowen-Walters [BW] and  Keynes-Sears [KS].

\begin{Proposition}\label{equivalence}
Let $\varphi_t$ be a continuous flow on a compact metric space $M$ without singularities. Then the following three conditions are equivalent:

$(1)$ For any $\epsilon>0$, there is $\delta>0$ such that for any $x,y\in M$ and any increasing continuous functions $\theta: \mathbb R\to \mathbb R$, if $d(\varphi_t(x),\varphi_{\theta(t)}(y))\le\delta$ for all $t\in\mathbb{R}$, then $\varphi_{\theta(t)}(y)\in\varphi_{[-\epsilon,\epsilon]}(\varphi_t(x))$ for all $t\in\mathbb{R}$;

$(2)$ For any $\epsilon>0$, there is $\delta>0$ such that for any $x,y\in M$ and any increasing continuous functions $\theta: \mathbb R\to \mathbb R$, if $d(\varphi_t(x),\varphi_{\theta(t)}(y))\le\delta$ for all $t\in\mathbb{R}$, then $\varphi_{\theta(0)}(y)\in\varphi_{[-\epsilon,\epsilon]}(x)$;

$(3)$ For any $\epsilon>0$, there is $\delta>0$ such that for any $x,y\in M$ and any
surjective increasing continuous functions $\theta: \mathbb R\to \mathbb R$,
if $d(\varphi_t(x),\varphi_{\theta(t)}(y))\le\delta$ for all $t\in\mathbb{R}$,
then $\varphi_{\theta(t_0)}(y)\in\varphi_{[-\epsilon,\epsilon]}(\varphi_{t_0}(x))$ for some $t_0\in R$;
\end{Proposition}

In case $\varphi_t$ is generated by a $C^1$ vector field $X$, item (1) of Proposition \ref{equivalence} is just the rescaled expansiveness because, in the non-singular case, $\|X(x)\|$ has an upper bound and also a positive lower bound. Item (3) is just  the expansiveness of Komuro [Kom1]. Item (2) is  the expansiveness of Bowen-Walter [BW] and Keynes-Sears [KS]. Thus Proposition \ref{equivalence} says that, for nonsingular flows, the three versions of expansiveness are equivalent.

 Note that for flows with singularities item (2) and item (3) are not equivalent: Komuro [Kom1] has proved that the geometrical Lorenz attractor satisfies item (3) but not item (2).

\vskip 0.4cm

{\noindent\it Proof.} That $(2)\Leftrightarrow (3)$ is proved by Oka [O]. Since $(1)\Rightarrow (2)$ is obvious, we only prove $(2)\Rightarrow (1)$.

Assume item $(2)$ holds. Let $\delta$ be chosen in item $(2)$ associated to $\epsilon$. Let $x,y\in M$ and any surjective increasing continuous functions $\theta: \mathbb R\to \mathbb R$ be given such that $d(\varphi_{\theta(t)}(y),\varphi_{t}(x))\le\delta$ for all $t\in\mathbb{R}$. For any $\tau\in\mathbb{R}$,  set
$$z=\varphi_\tau(x), ~ y_1=\varphi_{\theta(\tau)}(y), ~ \theta_1(t)=\theta(t+\tau)-\theta(\tau).$$ Then
$$d(\varphi_{\theta_1(t)}(y_1),\varphi_t(z))=d(\varphi_{\theta(t+\tau)}(y),\varphi_{t+\tau}(x))\le\delta.$$ Hence by item $(2)$ we have $y_1=\varphi_{\theta_1(0)}(y_1)\in\varphi_{[-\epsilon, ~\epsilon]}(z)$. Thus $$\varphi_{\theta(\tau)}(y)\in\varphi_{[-\epsilon, ~\epsilon]}(\varphi_{\tau}(x)).$$
This means item $(1)$ holds, proving Proposition 6.1.

\subsection*{Acknowledgement}
The first author is supported by National Natural Science
Foundation of China (No. 11671025 and No. 11571188)
and the Fundamental Research Funds for the Central Universities.
The second author is
supported by National Natural Science
Foundation of China (No. 11231001). We thank Shaobo Gan, Ming Li and Dawei Yang for many discussions and communications. We also thank Christian Bonatti and Adriana da Luz for an early preprint of their recent paper.

\end{document}